\documentclass[12pt]{article}
\usepackage[margin=1in]{geometry}
\usepackage{setspace}
\usepackage{amssymb}
\usepackage{amsmath}
\usepackage{amsthm}
\usepackage{hyperref}

\usepackage[ruled,vlined,linesnumbered]{algorithm2e}
\SetKwInput{KwInput}{Input} 
\newcommand{\TV}{\scriptsize\mathrm{TV}}

\newcommand{\df}{\mathrm{d}}
\newcommand{\Lik}{f}
\newcommand{\Z}{\mathsf{Z}}
\newcommand{\ind}{\mathbf{1}}
\newcommand{\Mtd}{k_{\scriptsize\mathrm{DA}}}

\newtheorem{theorem}{Theorem}[section]
\newtheorem{lemma}[theorem]{Lemma}

\newtheorem{corollary}[theorem]{Corollary}
\newtheorem{proposition}[theorem]{Proposition}

\title{\bf Convergence analysis of data augmentation algorithms in Bayesian lasso models with log-concave likelihoods}

\date{}

\author{Jingkai Cui \and Qian Qin}

\date{
School of Statistics \\ University of Minnesota}

\begin{document}

\maketitle

\begin{abstract}
    We study the convergence properties of a class of data augmentation algorithms targeting posterior distributions of Bayesian lasso models with log-concave likelihoods. 
    Leveraging isoperimetric inequalities, we derive a generic convergence bound for this class of algorithms and apply it to Bayesian probit, logistic, and heteroskedastic Gaussian linear lasso models.
    Under feasible initializations, the mixing times for the probit and logistic models are of order $O[(p+n)^3 (pn^{1-c} + n)]$, up to logarithmic factors, where $n$ is the sample size, $p$ is the dimension of the regression coefficients, and $c \in [0,1]$ is determined by the lasso penalty parameter.
    The mixing time for the heteroskedastic Gaussian model is $O[n(n+p)^3 (p n^{1-c} + n)]$, up to logarithmic factors.
\end{abstract}

\section{Introduction}

Lasso is an important tool in regression analysis for obtaining shrinkage estimators of regression coefficients \cite{tibshirani1996regression}.
In a Bayesian setting, an analogue of lasso is implemented by imposing Laplace (double exponential) priors on the regression coefficients \cite{figueiredo2003adaptive,park2008bayesian}.
The posterior distribution associated with a Bayesian lasso model, which has density of the form
\[
\pi(\beta) \propto \exp\left[ -\ell(\beta) + \lambda\|\beta\|_1 \right],
\]
where $\ell(\beta)$ is the negative log-likelihood, $\lambda$ is the lasso penalty parameter, and $\|\cdot\|$ is the $L^1$ norm (sum of the absolute values of a vector's components),
is typically intractable.
In this paper, we study a data augmentation algorithm, which is a type of Gibbs-like Markov chain Monte Carlo (MCMC) sampler \cite{tanner1987calculation,van2001art}, for sampling from this posterior distribution.
The main assumption is that the likelihood function of the model is log-concave.
We derive a quantitative convergence bound for this algorithm, and apply the bound to three specific models: a Bayesian probit lasso model, a Bayesian logistic lasso model, and a Bayesian Gaussian linear lasso model with heteroskedasticity.

The mixing properties of Gibbs-like algorithms for Bayesian lasso models, originated from \cite{park2008bayesian}, have been analyzed by \cite{khare2013geometric,rajaratnam2015mcmc,rajaratnam2015scalable,lee2024fast} in a homoskedastic linear regression setting when the likelihood is Gaussian.
In particular, \cite{lee2024fast} obtained quantitative bounds on the mixing time that are of the order $p^2 (n+p)^3$ (up to logarithm factors) when $\lambda$ is fixed, where $n$ is the sample size of the underlying data set, and $p$ is the dimension of the regression coefficients vector.
Our work borrows some techniques from \cite{lee2024fast}.

The mixing properties of Gibbs-like algorithms for Bayesian probit and logistic models with normal or flat priors, invented by \cite{albert1993bayesian} and \cite{polson2013bayesian} respectively, have been studied extensively.
Geometric ergodicity for these samplers are established in \cite{roy:hobe:2007}, \cite{choi2013polya}, and \cite{chakraborty2016convergence}.
In the context of probit regression, quantitative convergence bounds that scale favorably with $n$ and $p$ are established in \cite{qin2019convergence} and \cite{qin2020wasserstein} via coupling techniques.
\cite{johndrow2016mcmc} showed data augmentation algorithms for probit and logistic regression models mix slowly when the data set is imbalanced.
More recently, sharper mixing time bounds were obtained by \cite{lee2024fast} and \cite{ascolani2025mixing}.
See also \cite{wadia2024mixing}, who studied Gibbs samplers targeting strongly log-concave and smooth target distributions.

Finally, the Gaussian linear model with heteroskedasticity studied herein turns out to be equivalent to a linear model with Laplace errors.
Outside the context of Bayesian lasso, Gibbs-like algortithms related to this model were studied, largely qualitatively, by \cite{choi2013analysis} and \cite{roy2010monte} etc.

To carry out convergence analysis for Bayesian lasso models with log-concave likelihood functions, we rely on a technique based on conductance and isoperimetric inequalities.
The technique dates back to pioneering works of \cite{lawler1988bounds,jerrum1988conductance,dyer1991random,lovasz1993random}.
See \cite{chewi2025log} for an in-depth tutorial.
In the aforementioned works regarding regression models, it is utilized by \cite{wadia2024mixing}, \cite{lee2024fast} and \cite{ascolani2025mixing}.
The technique enjoyed tremendous recent success in the analysis of various MCMC algorithms targeting log-concave posterior distributions.
See, e.g., \cite{dwivedi2019log,wu2022minimax,andrieu2024explicit}.
Such methods have also been extended to some non-concave settings; see, e.g., \cite{vempala2019rapid,goyal2025mixing,zhou2025polynomial,kwon2024phase}.

Compared to previous works, the Bayesian model we study is the product of a possibly non-Gaussian (but log-concave) likelihood and a non-smooth (but log-concave) prior distribution.
This brings some unique challenges to the application of the aforementioned technique.
Taking inspiration from \cite{lee2024fast}, we utilize perturbation bounds from \cite{milman2009role} to establish an isoperimetric inequality for the log-concave but non-smooth posterior distribution.
This is then combined with a carefully constructed close coupling condition to bound the spectral gap of the underlying Markov chain.
By choosing a feasible initial distribution, one can then establish a quantitative mixing time bound.

For the Bayesian lasso and logistic models, it is found that when  $\lambda = \Omega(n^c)$ for some $c \in [0,1]$, the algorithm takes $O[(p+n)^3 (pn^{1-c} + n)]$ steps to mix (which is a common recommendation \cite{zhao2006model,van2008high}).
For the heteroskedastic Gaussian model, the mixing time is of order $O[n (n+p)^3 (pn^{1-c} +n)]$.

The rest of this article is organized as follows.
In Section \ref{sec:da}, we define the Bayesian lasso model and the data augmentation algorithm in question.
In Section \ref{sec:convergence}, we state a convergence bound for the data augmentation algorithm in a generic setting, and apply the bound to the three specific models mentioned previously.
Section \ref{sec:proof} contains the derivation of the generic convergence bound.

\section{Data Augmentation Algorithm for Bayesian Lasso} \label{sec:da}

\subsection{Bayesian lasso regression}

Throughout, for a vector $a$ of length $k$, we use $a_i$, $1 \leq i \leq k$, to denote its $i$th component.

Assume that $Y = (Y_1, \dots, Y_n)^{\top}$ consists of observable response variables, where $n$ is a positive integer, and each $Y_i$ takes values in a common $\sigma$-finite measure space.
Let $A \in \mathbb{R}^d$ and $B \in \mathbb{R}^p$ be (column) vectors of unknown regression coefficients, where $d$ and $p$ are positive integers.
As will become clear below, the coefficients in $B$ are those for which sparsity is to be encouraged, whereas the coefficients in $A$ are treated as non-sparse.
We consider Bayesian lasso models of the form
\begin{equation} \label{eq:model}
\begin{aligned}
    Y \mid A, B &\sim \Lik(\cdot \mid A, B), \\
    B_j \mid A & \stackrel{\text{ind}}{\sim} \text{Laplace}(\lambda), \quad j = 1, \dots, p, \\
    A_j & \stackrel{\text{ind}}{\sim} \mathrm{N}(0, \theta^{-2}), \quad j = 1,\dots,d.
\end{aligned}
\end{equation}
Here, for $\alpha \in \mathbb{R}^d$ and $\beta \in \mathbb{R}^p$, $\Lik(\cdot \mid \alpha, \beta)$ is a probability density function designating the likelihood of $Y$; $\lambda$ and $\theta$ are known positive hyperparameters; $\text{Laplace}(\lambda)$ is the distribution with density function proportional to $u \mapsto e^{-\lambda|u|}$; finally, $\mbox{N}(0,\theta^{-2})$ is the normal distribution with mean 0 and variance $\theta^{-2}$.
Note that $A$ and $B$ are a priori independent.

The correspondence between this Bayesian model and the frequentist lasso model is well-known.
In particular, given an observed dataset $y$, the posterior mode of the Bayesian model is the minimizer of the lasso loss 
\[
-\log \Lik(y \mid \alpha, \beta) + \frac{\theta^2 \|\alpha\|_2^2}{2} + \lambda \|\beta\|_1,
\]
where $\|\cdot\|_2$ is the Euclidean norm.
That is, in the frequentist analogue, an $L^2$ (ridge) penalty is placed on $A$, while an $L^1$ (lasso) penalty is placed on $B$.

The posterior density function of $(A,B)$ given $Y = y$ is
\begin{equation} \label{eq:posterior}
\pi_{A, B \mid Y}(\alpha, \beta \mid y) \propto \Lik(y \mid \alpha, \beta) \, \exp \left( - \frac{\theta^2 \|\alpha\|_2^2}{2} -\lambda \|\beta\|_1  \right).
\end{equation}
Typically, this is an intractable probability distribution.
In the next subsection, we describe a data augmentation algorithm for sampling from $\pi_{A,B \mid Y}(\cdot \mid y)$.

\subsection{Data augmentation algorithm}

The data augmentation algorithm in consideration relies on an auxillary random element $Z$ taking values in some $\sigma$-finite measure space $\Z$.
Suppose that $g(\cdot \mid \alpha, \beta, y)$ is a probability density function on~$\Z$.
Combined with a well-known Gaussian mixture representation of Laplace distributions \cite{andrews1974scale}, one arrives at the following augmented model:
\begin{equation} \label{eq:model_1}
    \begin{aligned}
        Z \mid A, B, T, Y &\sim g(\cdot \mid A, B, Y),\\
        Y \mid A, B, T &\sim \Lik(\cdot \mid A, B), \\
        B_j \mid A, T & \stackrel{\text{ind}}{\sim} \mbox{N}(0, T_j), \quad j = 1, \dots, p, \\
        A_j \mid T &\stackrel{\text{ind}}{\sim} \mbox{N}(0, \theta^{-2}), \quad j = 1,\dots, d, \\
        T_j & \stackrel{\text{ind}}{\sim} \mbox{Exp}(\lambda^2/2), \quad j = 1, \dots, p,
    \end{aligned}
\end{equation}
where $\mbox{Exp}(\lambda^2/2)$ has density function proportional to $u \mapsto e^{-\lambda^2 u / 2} \, \ind(u > 0)$, with $\ind(\cdot)$ denoting indicator functions.
Here, $T \in \mathbb{R}^p$ is an auxiliary random vector ensuring that $B_j \sim \mbox{Laplace}(\lambda)$ independently a priori \cite{andrews1974scale}.
Then the conditional distribution of $(A,B)$ given $Y = y$ derived from \eqref{eq:model_1} is still of the form \eqref{eq:posterior}.

Based on \eqref{eq:model_1}, one can construct a data augmentation algorithm targeting the posterior density $\pi_{A, B \mid Y}(\cdot \mid y)$ by iteratively drawing from the following conditional densities:
\[
\begin{aligned}
    \pi_{A, B \mid Z, T, Y}(\alpha, \beta \mid z, \tau, y) &\propto g(z \mid \alpha, \beta, y) \, \Lik(y \mid \alpha, \beta) \left[ \prod_{j=1}^d \exp \left( - \frac{\theta^2 \alpha_j^2}{2} \right) \right] \left[ \prod_{j=1}^p \exp \left( - \frac{\beta_j^2}{2 \tau_j} \right) \right], \\
    \pi_{Z, T \mid A, B, Y}(z, \tau \mid \alpha, \beta, y) &\propto g(z \mid \alpha, \beta, y) \, \prod_{j=1}^p \tau_j^{-1/2} \exp \left( - \frac{\beta_j^2}{2 \tau_j} - \frac{\lambda^2 \tau_j}{2} \right) \ind(\tau_j > 0).
\end{aligned}
\]
For the data augmentation scheme to work, $g(\cdot \mid \alpha, \beta, y)$ must be designed in a manner such that one can efficiently sample from $\pi_{A, B \mid Z, T, Y}(\cdot \mid z, \tau, y)$ and $g(\cdot \mid \alpha, \beta, y)$.
Then one can implement Algorithm \ref{alg:da}, which simulates a Markov chain $(A(t), B(t))_{t=0}^{\infty}$ with transition density
\[
    \Mtd((\alpha, \beta), (\alpha', \beta')) = \int_{\Z \times (0,\infty)^p} \pi_{A, B \mid Z, T, Y}(\alpha', \beta' \mid z, \tau, y) \, \pi_{Z,T \mid A, B, Y}(z, \tau \mid \alpha, \beta, y) \, \df (z, \tau) .
\]
The underlying Markov chain is reversible with respect to $\pi_{A,B \mid Y}(\cdot \mid y)$.
Under mild regularity conditions (see, e.g., \cite{tierney1994markov}), the law of $(A(t), B(t))$ converges to the posterior distribution as $t \to \infty$.
The question is can we quantify the rate of convergence.

\bigskip

\begin{algorithm}[H]
\caption{Data augmentation algorithm, $t$'th iteration} \label{alg:da}
\KwInput{The current state $(A(t), B(t))$}

\For{$j = 1,\dots,p$}{
    Draw $1/\tau_j$ from $\text{InvGaussian}(\lambda/|B_j(t)|, \lambda^2)$, which has density function
    \[
    u \mapsto \frac{\lambda}{\sqrt{2 \pi}} \, \exp[\lambda |B_j(t)|] \, u^{-3/2} \exp \left[ - \frac{B_j(t)^2 u}{2} - \frac{\lambda^2}{2u} \right] \ind(u > 0) ;
    \] 
    (Note this distribution is well-defined even if $B_j(t) = 0$.)
}

Draw $z$ from $g(\cdot \mid A(t), B(t), y)$\;

Draw $(A(t+1), B(t+1))$ from $\pi_{A, B \mid Z, T, Y}(\cdot \mid z, \tau, y)$\;

\end{algorithm}

\bigskip

\subsection{Special cases} \label{ssec:special}

Algorithm \ref{alg:da} encompasses several important special cases, as we now describe.

Let $x_i$, $i = 1,\dots, n$, be a sequence of $\mathbb{R}^p$-valued covariates.
Let $\phi(\cdot)$ and $\Phi(\cdot)$ be the density and distribution functions of the standard normal distribution, respectively.
Then Algorithm~\ref{alg:da} has the following special forms when $d = 1$.

{\it Probit lasso:}
Let
\begin{equation} \label{eq:lik-probit}
\Lik(y \mid \alpha, \beta) = \prod_{i=1}^n \Phi(\alpha + x_i^{\top} \beta)^{y_i} [1 - \Phi(\alpha + x_i^{\top} \beta)]^{1-y_i}, \quad y \in \{0,1\}^n,
\end{equation}
so \eqref{eq:model} is a Bayesian probit model with lasso (and ridge) penalty.
In other words,
\[
Y_i \mid A, B \stackrel{\mathrm{ind}}{\sim} \mbox{Bernoulli}(\Phi(A + x_i^{\top} B)), \quad i = 1,\dots, n.
\]
Following \cite{albert1993bayesian}, for $z \in \Z = \mathbb{R}^n$, let
\begin{equation} \label{eq:g-probit}
\begin{aligned}
    & g(z \mid \alpha, \beta, y) \\
    =& \prod_{i=1}^n \left[  \frac{\phi(z_i - \alpha - x_i^{\top} \beta)}{\Phi(\alpha + x_i^{\top} \beta)} \, \ind(z_i \geq 0) \, \ind(y_i = 1) + \frac{\phi(z_i - \alpha - x_i^{\top} \beta)}{1-\Phi(\alpha + x_i^{\top} \beta)} \, \ind(z_i < 0) \, \ind(y_i = 0) \right].
\end{aligned}
\end{equation}
In other words, if an $n$-dimensional random vector $Z'$ is distributed as $g(\cdot \mid \alpha, \beta, y)$, then $Z'_1, \dots, Z'_n$ are independent;
moreover, $Z'_i$ follows a normal distribution with mean $\alpha + x_i^{\top} \beta$ and unit variance, truncated to $[0,\infty)$ if $y_i = 1$, and to $(-\infty,0)$ if $y_i = 0$. 
As shown in \cite{albert1993bayesian}, $\pi(\alpha, \beta \mid z, \tau, y)$ then corresponds to the $(1+p)$-dimensional normal distribution with mean $[X^{\top} X + D(\tau)]^{-1} X^{\top} z$ and covariance $[X^{\top} X + D(\tau)]^{-1}$, where $X$ is the $n \times (1+p)$ design matrix whose $i$th row is $(1, x_i^{\top})$, and $D(\tau) = \mbox{diag}(\theta^2, 1/\tau_1, \dots, 1/\tau_p)$.

{\it Logistic lasso:}
Let
\begin{equation} \label{eq:lik-logistic}
\Lik(y \mid \alpha, \beta) = \prod_{i=1}^n \frac{\exp[ y_i (\alpha + x_i^{\top} \beta)]}{1 + \exp(\alpha + x_i^{\top} \beta)}, \quad y \in \{0,1\}^n,
\end{equation}
so \eqref{eq:model} is a Bayesian logistic model with lasso penalty.
In other words,
\[
Y_i \mid A, B \stackrel{\mathrm{ind}}{\sim} \mbox{Bernoulli} \left( \frac{\exp (A + x_i^{\top} B)}{1 + \exp (A + x_i^{\top} B)} \right), \quad i = 1,\dots, n.
\]
Following \cite{polson2013bayesian}, let $g(\cdot \mid \alpha, \beta, y)$ correspond to independent P\'{o}lya-Gamma (PG) random variables: for $z \in \Z = (0,\infty)^n$,
\begin{equation} \label{eq:g-logistic}
g(z \mid \alpha, \beta, y) = \prod_{i=1}^n \cosh \left( \frac{\alpha + x_i^{\top} \beta}{2} \right) \exp\left[ -\frac{z_i \, (
\alpha + x_i^{\top} \beta)^2}{2} \right]  f_{\scriptsize\mbox{PG}}(z_i),
\end{equation}
where $f_{\scriptsize\mbox{PG}}(\cdot)$ is is the density of the so called ``$\mbox{PG}(1,0)$" distribution.
The $i$th factor in \eqref{eq:g-logistic} corresponds to what is called the $\mbox{PG}(1, |\alpha + x_i^{\top} \beta|)$ distribution.
See \cite{polson2013bayesian} for the exact definition of PG distributions and ways to sample from them.
As shown in \cite{polson2013bayesian}, with this choice of $g(\cdot \mid \alpha, \beta, y)$, the density $\pi(\alpha, \beta \mid z, \tau, y)$ corresponds to the $(1+p)$-dimensional normal distribution with mean $[X^{\top} \Lambda X + D(\tau)]^{-1} X^{\top} (y - \ind_n / 2)$ and covariance matrix $[X^{\top} \Lambda X + D(\tau)]^{-1}$, where $\Lambda = \mbox{diag}(z_1, \dots, z_n)$, and $\ind_n$ is the $n$-dimensional vector full of 1's.

{\it Gaussian linear lasso with heteroskedasticity:}
Suppose that, given $(A, B)$, the random vector $(Y,Z) \in \mathbb{R}^n \times \mathbb{R}^n$ is distributed as follows:
\[
\begin{aligned}
    Y_i \mid Z, A, B &\stackrel{\scriptsize\mathrm{ind}}{\sim} \mbox{N}(A + x_i^{\top} B,\, Z_i^{-1}), \quad i = 1, \dots, n, \\
    Z_i \mid A, B & \stackrel{\scriptsize \mathrm{ind}}{\sim} \mathrm{InvGamma}(1, \gamma^2/2), \quad i = 1, \dots, n,
\end{aligned}
\]
where $\gamma > 0$ is a hyperparameter, and $\mathrm{InvGamma}(c_1, c_2)$ with $c_1 > 0$ and $c_2 > 0$ has density function proportional to $u \mapsto u^{-c_1-1} e^{-c_2/u} \ind (u > 0)$.
In particular, given $(Z,A,B)$, the response variables $(Y_i)_{i=1}^n$ are normally distributed with different variances, so \eqref{eq:model} is a Bayesian Gaussian linear lasso model with heteroskedasticity.
In terms of \eqref{eq:model_1}, one can derive that
\begin{equation} \label{eq:g-gaussian}
\begin{aligned}
    \Lik(y \mid \alpha, \beta) =& \left(\frac{\gamma}{2}\right)^n \exp \left( - \gamma \sum_{i=1}^n |y_i - \alpha - x_i^{\top} \beta| \right), \\
    g(z \mid \alpha, \beta, y) =& \prod_{i=1}^n \frac{\gamma}{\sqrt{2 \pi}} \exp \left( \gamma \, |y_i - \alpha - x_i^{\top} \beta| \right) \times \\ &\quad \quad z_i^{-3/2} 
    \exp \left[ - \frac{z_i}{2} (y_i - \alpha - x_i^{\top} \beta)^2 - \frac{\gamma^2}{2z_i} \right] \, \ind(z_i > 0).
\end{aligned}
\end{equation}
In particular, $g(z \mid \alpha, \beta, y)$ corresponds to the product of $\mathrm{InvGaussian}(\gamma/|y_i - \alpha - x_i^{\top} \beta| ,\gamma^2)$ distributions.
Moreover, it is not difficult to see that $\pi_{A,B \mid Z, T, Y}(\cdot \mid z, \tau, y)$ corresponds to the $(1+p)$-dimensional normal distribution with mean $[X^{\top}\Lambda X + D(\tau)]^{-1} X^{\top} \Lambda y$ and covariance matrix $[X^{\top} \Lambda X + D(\tau)]^{-1}$.

\section{Convergence Bounds} \label{sec:convergence}

\subsection{Preliminaries}

In this subsection, we briefly recall some basic concepts regarding the convergence rates of Markov chains.

Abusing notations, we do not differentiate between probability distributions and their density functions.
For two probability density functions $g_1$ and $g_2$ defined on the same measure space $(\Omega, \mathcal{F}, \mu)$, their total variation distance is
\[
\|g_1(\cdot) - g_2(\cdot) \|_{\TV} = \frac{1}{2} \int_{\Omega} |g_1(x) - g_2(x)| \, \mu(\df x).
\]

Suppose that a Markov chain $(A(t), B(t))_{t=0}^{\infty}$ evolves according to the transition kernel $\Mtd$, with $(A(0),B(0))$ following an initial distribution with density function $(\alpha, \beta) \mapsto \omega(\alpha, \beta)$.
For a non-negative integer~$t$, denote by $\omega \, \Mtd^t (\cdot)$ the probability density function of $(A(t), B(t))$, i.e., $\omega \, \Mtd^0 (\cdot) = \omega (\cdot)$,
\[
\omega \, \Mtd^{t+1} (\cdot) = \int_{\mathbb{R}^{d+p}} \omega \, \Mtd^t(\alpha, \beta) \, \Mtd((\alpha, \beta), \cdot) \, \df (\alpha, \beta).
\]

We say $\omega \, \Mtd^t(\cdot)$ converges to $\pi_{A,B \mid Y} (\cdot \mid y)$ at a geometric rate if there exists some function of $\omega$, say $\tilde{C}(\omega)$, and a number $\rho \in [0,1)$, independent of $\omega$, such that, for $t \geq 1$,
\begin{equation} \label{ine:geo}
\|\omega \, \Mtd^t(\cdot) - \pi_{A,B \mid Y} (\cdot \mid y)\|_{\TV} \leq \tilde{C}(\omega) \, \rho^t,
\end{equation}

The smallest $t$ for which $\|\omega \, \Mtd^t(\cdot) - \pi_{A,B \mid Y} (\cdot \mid y)\|_{\TV}$ is no greater than some prescribed $\bar{\epsilon} > 0$ is called the $\bar{\epsilon}$-mixing time, and is denoted by $t(\omega, \bar{\epsilon})$.
When \eqref{ine:geo} holds,
\[
t(\omega, \bar{\epsilon}) \leq \frac{\log \tilde{C}(\omega) - \log \bar{\epsilon}}{- \log \rho}.
\]

\subsection{Generic convergence bound} \label{ssec:general}

We establish a convergence bound for Algorithm \ref{alg:da} under the following conditions:

\begin{itemize}
    \item [(A1)] \phantomsection\label{A1}
    There exist positive numbers $\delta$ and $\epsilon$ such that, for $\alpha^{(1)} \in \mathbb{R}^d$, $\alpha^{(2)} \in \mathbb{R}^d$, $\beta^{(1)} \in \mathbb{R}^p$, $\beta^{(2)} \in \mathbb{R}^p$ satisfying $\|\alpha^{(1)} - \alpha^{(2)} \|_2^2 + \lambda^2 \|\beta^{(1)} - \beta^{(2)}\|_2^2 < \delta^2$,
    \[
    \begin{aligned}
        \left\| g(\cdot \mid \alpha^{(1)}, \beta^{(1)}, y) - g(\cdot \mid \alpha^{(2)}, \beta^{(2)}, y) \right\|_{\TV}  + \sqrt{2} \, p^{1/4} \, \delta^{1/2}  \leq 1 - \epsilon.
    \end{aligned}
    \]

    \item [(A2)] \phantomsection\label{A2} The likelihood $\Lik(y \mid \alpha, \beta)$ and its negative logarithm, $\ell(\alpha, \beta) = - \log \Lik(y \mid \alpha, \beta)$, satisfy the following:
    \begin{enumerate}
        \item [(i)] The function $\ell(\alpha, \beta)$ is convex.
        \item [(ii)] The function $\ell(\alpha, \beta)$ can be approximated point-wisely from above by a sequence of convex functions that are twice differentiable.
        \item [(iii)] there exists a positive number $D$ (that may depend on $y$, $\lambda$, and $\theta$) such that
        \[
        \frac{\pi_{A, B\mid Y}(\alpha, \beta \mid y)}{\mu_{\theta, \lambda}(\alpha, \beta)} \leq e^D,
        \]
        where
        \[
            \mu_{\theta, \lambda}(\alpha, \beta) = \frac{\theta^d}{(2\pi)^{d/2}} \exp \left( -\frac{\theta^2 \|\alpha\|_2^2}{2} \right) \, \frac{\lambda^p}{2^p} \, \exp \left( - \lambda \|\beta\|_1 \right), \quad \alpha \in \mathbb{R}^d, \; \beta \in \mathbb{R}^p.
        \]
        
    \end{enumerate}
    
\end{itemize}

Condition \hyperref[A1]{(A1)} is called a ``close coupling," and is used to argue $k((\alpha^{(1)}, \beta^{(1)}), \cdot)$ and $k((\alpha^{(2)}, \beta^{(2)}), \cdot)$ have a significant overlap if $\alpha^{(1)}$ is close to $\alpha^{(2)}$ and $\beta^{(1)}$ is close to $\beta^{(2)}$.
To establish this condition, one would need to upper bound $\left\| g(\cdot \mid \alpha^{(1)}, \beta^{(1)}, y) - g(\cdot \mid \alpha^{(2)}, \beta^{(2)}, y) \right\|_{\TV}$.
This type of calculation can naturally arise in a study of the simpler data augmentation algorithm with transition density
\begin{equation} \label{eq:k}
k((\alpha, \beta), (\alpha', \beta')) = \int_{\Z} \pi_{A,B \mid Z, T, Y}(\alpha', \beta' \mid z, \tau, y) \, g(z \mid \alpha, \beta, y) \, \df z,
\end{equation}
where $\tau \in (0,\infty)^p$ is a constant.
The transition law \eqref{eq:k} is associated with a non-lasso version of the model \eqref{eq:model}, where $B_j$ is apriori distributed as $\mbox{N}(0, \tau_j)$ for a fixed and known $\tau_j$, as opposed to a Laplace distribution.
When $\Lik(y \mid \alpha, \beta)$ corresponds to a probit or logistic model, \cite{lee2024fast} studied the transition law \eqref{eq:k}.
We are able to recycle their bounds on $\left\| g(\cdot \mid \alpha^{(1)}, \beta^{(1)}, y) - g(\cdot \mid \alpha^{(2)}, \beta^{(2)}, y) \right\|_{\TV}$ in these settings.

Condition \hyperref[A2]{(A2)} indicates that $\pi_{A,B \mid Y}(\cdot \mid y)$ is a perturbed version of $\mu_{\theta, \lambda}(\cdot)$.
When $\pi_{A,B \mid Y}(\cdot \mid y)$ is log-concave, this implies the two distributions are similar in terms of isoperimetry \cite{milman2009role,milman2012properties}.
Isoperimetry concerns how the probability mass of a measurable set controls the amount of additional mass gained under small enlargements of the set, via the measure of its boundary.
One may then obtain an isoperimetric inequality for $\pi_{A,B \mid Y}(\cdot \mid y)$ based on a known isoperimetric ineqaulity for $\mu_{\theta, \lambda}(\cdot)$.
This approach was taken by \cite{lee2024fast} to study a data augmentation algorithm associated with a homoskedastic Gaussian lasso model.

Combining the close coupling and the isoperimetric inequality (see \cite{chewi2025log} or \cite{andrieu2024explicit} for an introduction to the general technique), we establish the following theorem.
The proof will be presented in Section \ref{sec:proof}.

\begin{theorem} \label{thm:main}
    Suppose that \hyperref[A1]{(A1)} and \hyperref[A2]{(A2)} hold.
    Then \eqref{ine:geo} holds with
    \[
    \rho = 1 - \frac{1}{32} \epsilon^2 \min \left\{ 1, \frac{1}{4} \, \delta^2 \, \frac{C_1^2 \min \{ 1, \theta^2 \}}{(D+1)^2} \right\},
    \]
    where $C_1$ is a positive universal constant, and
    \begin{equation} \label{eq:C-omega}
    \tilde{C}(\omega) = \frac{1}{2} \sqrt{\int_{\mathbb{R}^{d+p}} \left[ \frac{\omega(\alpha, \beta)}{\pi_{A,B \mid Y}(\alpha, \beta \mid y)} - 1 \right]^2 \pi_{A, B \mid Y}(\alpha, \beta \mid y) \, \df (\alpha, \beta)} \, .
    \end{equation}
\end{theorem}

\subsection{Bounds for probit, logistic, and heteroskedastic Gaussian lasso}

We now apply Theorem \ref{thm:main} to the three regression models in Section \ref{ssec:special}.

We will study how $\rho$ in \eqref{ine:geo} and $t(\omega, \bar{\epsilon})$ scale with $n$, $d$, and $p$.
To this end, imagine that there is a sequence of data sets $(X_{(N)},y_{(N)})_{N=1}^{\infty}$, where $X_{(N)}$ is a design matrix and $y_{(N)}$ is a response vector.
For a given $N$, we suppress the subscript, so $X_{(N)}$ and $y_{(N)}$ are just $X$ and $y$ as defined in Section \ref{ssec:special}.
Each data set is associated with some value of $(n, p, \theta, \lambda)$, where $n$ and/or $p$ grow with~$N$.
For two sequences of positive numbers indexed by $N$, say, $(a_{(N)})_{N=1}^{\infty}$ and $(b_{(N)})_{N=1}^{\infty}$, write $b = O(a)$ if $b_{(N)}/a_{(N)}$ is bounded, and $b = \Omega(a)$ if $b_{(N)}/a_{(N)}$ is bounded away from zero as $N \to \infty$.

Let $X_{\lambda} \in \mathbb{R}^{n \times (1+p)}$ be the matrix whose $i$th row is $(1, x_i^{\top} / \lambda)$.
One has the decomposition
\[
     X_{\lambda}^{\top} X_{\lambda} = \frac{X^{\top}X}{\lambda} + \left( \begin{array}{cc} (1 - \lambda^{-1}) n & 0 \\ 0 & (\lambda^{-2} - \lambda^{-1}) \sum_{i=1}^n x_i x_i^{\top} \end{array}\right).
\]
Denote by $\sigma_{\max}(\cdot)$ the largest eigenvalue of a matrix.
Then, when $\lambda \geq 1$, by Weyl's inequality (see, e.g., Theorem 4.3.1 of \cite{horn2012matrix}),
\begin{equation} \label{ine:X-lambda}
     \sigma_{\max}(X_{\lambda}^{\top} X_{\lambda}) \leq \frac{\sigma_{\max}(X^{\top}X)}{\lambda} + \frac{(\lambda - 1) n}{\lambda} \leq  \sigma_{\max}(X^{\top}X).
\end{equation}

The next proposition gives asymptotic bounds on~$\rho$ for each model.

\begin{proposition} \label{pro:rho-examples}
    Each of the following statements holds.

    \begin{enumerate}
        \item Consider the data augmentation algorithm for Bayesian probit lasso model.
        Then \eqref{ine:geo} holds with $\rho$ satisfying
        \begin{equation} \label{eq:rho-scale-probit}
        1 - \rho = \Omega\left[ \frac{\min \{1, \theta^2\}}{\max\{ \sigma_{\max}(X_{\lambda}^{\top}X_{\lambda}), p \} \, (n + p)^2 \, (\log M)^2 } \right],
        \end{equation}
        where $M = (\theta^{-2} + 1) \, \sigma_{\max}(X_{\lambda}^{\top}X_{\lambda})$.

        \item Consider the data augmentation algorithm for Bayesian logistic lasso model.
        Then \eqref{ine:geo} holds with $\rho$ also satisfying \eqref{eq:rho-scale-probit}.

        \item Consider the data augmentation algorithm for Bayesian Gaussian lasso model with heteroskedasticity.
        Assume that $\gamma = O(1)$, and that $\sum_{i=1}^n |y_i| = O(n)$.
        Then \eqref{ine:geo} holds with $\rho$ satisfying
        \[
            1 - \rho = \Omega \left[ \frac{\min\{1, \theta^2\}}{\max\{ n \sigma_{\max}(X_{\lambda}^{\top} X_{\lambda}), p \} \, (n + p)^2 \, (\log M)^2 }  \right].
        \]
    \end{enumerate}
    In all three settings, one can take $\tilde{C}(\omega)$ as in \eqref{eq:C-omega}.
\end{proposition}

\begin{proof}
    Proposition \ref{pro:rho-examples} consists of special cases of Theorem \ref{thm:main}.
    Theorem \ref{thm:main} gives the formula of $\rho$ and $\tilde{C}(\omega)$ under two generic conditions, \hyperref[A1]{(A1)} and \hyperref[A2]{(A2)}.
    These two conditions are verified for each of the three models in Section \ref{ssec:a-special}.
    In particular, \hyperref[A1]{(A1)} is verified in Lemma \ref{lem:close-examples}, and \hypertarget{A2}{(A2)} is verified in Lemma \ref{lem:iso-exmaples}.
\end{proof}

From Proposition \ref{pro:rho-examples}, one can immediately obtain asymptotic bounds on $t(\omega, \bar{\epsilon})$ through the formula
\[
t(\omega, \bar{\epsilon}) = O\left[\frac{\log \tilde{C}(\omega) - \log \bar{\epsilon}}{1-\rho} \right]
\]
if $\log \tilde{C}(\omega)$ can be bounded.

The warmness parameter $\tilde{C}(\omega)$ measures how far away the initial distribution $\omega(\cdot)$ is from the target distribution $\pi_{A,B \mid y}(\cdot \mid y)$.
We now propose, for each model, a  feasible initial density $\omega(\cdot)$, and bound $\log \tilde{C}(\omega)$.
The following proposition is proved in Appendix \ref{app:proofs}.

\begin{proposition} \label{pro:initial}
    For $\eta \in \mathbb{R}^{1+p}$, $L \in [0,\infty)$, $\alpha \in \mathbb{R}$, and $\beta \in \mathbb{R}^p$, let
    \[
    \begin{aligned}
        \omega_{\eta, L}(\alpha, \beta) = & \left( \frac{L + \theta^2}{2\pi} \right)^{1/2} \left( \frac{L + 1}{2\pi} \right)^{p/2} \lambda^p \\
        & \quad \quad \exp \left\{ - \frac{1}{2} \left[ \left( \begin{array}{c} \alpha \\ \lambda \beta \end{array} \right) + V_L \eta \right]^{\top}  V_L^{-1}  \left[ \left( \begin{array}{c} \alpha \\ \lambda \beta \end{array} \right) + V_L \eta \right]  \right\},
    \end{aligned}
    \]
    where 
    \[
    V_L = \left( \begin{array}{cc}
        1/(L + \theta^2) & 0  \\
        0 & I_p/(L+1)
    \end{array} \right),
    \]
    where $I_p$ is the $p \times p$ identity matrix.
    That is, $\omega_{\eta, L}(\cdot)$ is the density function of $(A_*^{\top}, B_*^{\top})^{\top}$, where $(A_*^{\top}, \lambda B_*^{\top})^{\top}$ is normally distributed with mean $-V_L \eta$ and variance $V_L$.
    Then each of the following statements holds.

    \begin{enumerate}
        \item In the context of Bayesian probit lasso model, take 
        \[
        \eta = -\sqrt{\frac{2}{\pi}} \sum_{i=1}^n (2y_i - 1) \left( \begin{array}{c} 1 \\ x_i/\lambda \end{array} \right), \quad L = \sigma_{\max}(X_{\lambda}^{\top} X_{\lambda}).
        \]
        Then, with $\tilde{C}(\omega)$ defined in \eqref{eq:C-omega}, one has
        \begin{equation} \label{eq:C-omega-examples}
        \log \tilde{C}(\omega_{\eta, L}) = O\left( \log M'' + p \log M' + n \right),
        \end{equation}
        where $M' = \sigma_{\max}(X_{\lambda}^{\top} X_{\lambda}) + 1$ and $M'' = \sigma_{\max}(X_{\lambda}^{\top} X_{\lambda})/\theta^2 + 1$.

        \item In the context of Bayesian logistic model, take
        \[
        \eta = -\sum_{i=1}^n \frac{2y_i - 1}{2} \left( \begin{array}{c} 1 \\ x_i/\lambda \end{array} \right), \quad L = \frac{1}{4} \sigma_{\max}(X_{\lambda}^{\top} X_{\lambda}).
        \]
        Then \eqref{eq:C-omega-examples} holds.

        \item In the context of Bayesian Gaussian model with heteroskedasticity, take
        \[
        \eta = 0, \quad L = \gamma \, \sigma_{\max}(X_{\lambda}^{\top} X_{\lambda}). 
        \]
        Assume that $\gamma = O(1)$, and that $\sum_{i=1}^n |y_i| = O(n)$.
        Then \eqref{eq:C-omega-examples} holds.
    \end{enumerate}
\end{proposition}

Recommendations in the literature typically suggest taking $\lambda$ to be $O(n)$ and $\Omega(n^{1/2})$ \cite{zhao2006model,van2008high}.
On the other hand, it is common to assume $\sigma_{\max}(X^{\top}X) = O(pn)$ \cite{johnstone2001distribution,lee2024fast}.
Then we have the following more explicit asymptotic bounds on the mixing time.

\begin{corollary}
    Assume that $\lambda = \Omega(n^c)$ for some constant $c \in [0,1]$, $\sigma_{\max}(X^{\top}X) = O(pn)$, and that $\theta = \Omega(1)$.
    Let $(\bar{\epsilon}_N)_{N=1}^{\infty}$ be a sequence of positive numbers.
    Then each of the following holds:
    \begin{enumerate}
        \item For the Bayesian probit lasso model, when the initial distribution $\omega(\cdot)$ is given in Proposition \ref{pro:initial},
        \begin{equation} \label{t-probit}
        t(\omega, \bar{\epsilon}) = O \left\{ (p \log M_1 +n - \log \bar{\epsilon}) [(pn^{1-c} + n)(n+p)^2 (\log M_1)^2]  \right\},
        \end{equation}
        where $M_1 = pn^{1-c} + n$.

        \item  For the Bayesian logistic lasso model, when the initial distribution $\omega(\cdot)$ is given in Proposition \ref{pro:initial}, the mixing time satisfies \eqref{t-probit}.

        \item For the heteroskedastic Gaussian model, assume further that $\gamma = O(1)$, and that $\sum_{i=1}^n |y_i| = O(n)$.
        Then, when the initial distribution $\omega(\cdot)$ is given in Proposition \ref{pro:initial}, 
        \[
        t(\omega, \bar{\epsilon}) = O \left\{ (p \log M_1 +n - \log \bar{\epsilon}) [n (pn^{1-c} + n)(n+p)^2 (\log M_1)^2]  \right\}.
        \]
    \end{enumerate}
\end{corollary}

\begin{proof}
    By \eqref{ine:X-lambda}, $\sigma_{\max}(X_{\lambda}^{\top} X_{\lambda}) = O(p n^{1-c} + n)$.
    The desired result then follows from Propositions \ref{pro:rho-examples} and \ref{pro:initial}. 
\end{proof}

\subsection{Conditions \hyperref[A1]{(A1)} and \hyperref[A2]{(A2)} in special cases} \label{ssec:a-special}

In this subsection, we establish the conditions in Theorem \ref{thm:main} in the context of probit, logistic, and heteroskedastic Gaussian lasso models.
Proposition \ref{pro:rho-examples} then follows.

In the context of probit and logistic lasso, we can utilize existing results from \cite{lee2024fast} to derive \hyperref[A1]{(A1)}.
In the context of heteroskedastic Gaussian lasso, one can establish \hyperref[A1]{(A1)} using Pinsker's inequality.

\begin{lemma} \label{lem:close-examples}
    Each of the following statements holds.
    \begin{enumerate}
        \item In the data augmentation algorithm for Bayesian probit lasso model, $g(\cdot \mid \alpha, \beta, y)$ has the form \eqref{eq:g-probit}.
        Then \hyperref[A1]{(A1)} holds with $\epsilon = 1/2$ and 
        \[
            \delta = \min \left\{ \frac{1}{2 \sqrt{\sigma_{\max}(X_{\lambda}^{\top}X_{\lambda})} } , \, \frac{1}{32 \sqrt{p}} \right\}.
        \]

        \item In the data augmentation algorithm for Bayesian logistic lasso model, $g(\cdot \mid \alpha, \beta, y)$ has the form \eqref{eq:g-logistic}.
        Then \hyperref[A1]{(A1)} holds with $\epsilon = 1/2$ and
        \[
            \delta = \min \left\{ \frac{1}{\sqrt{\sigma_{\max}(X_{\lambda}^{\top}X_{\lambda})} } , \, \frac{1}{32 \sqrt{p}} \right\}.
        \]

        \item In the data augmentation algorithm for Bayesian Gaussian lasso model with heteroskedasticity, $g(\cdot \mid \alpha, \beta, y)$ is given in \eqref{eq:g-gaussian}.
        Then \hyperref[A1]{(A1)} holds with $\epsilon = 1/2$ and
        \[
            \delta = \min \left\{ \frac{1}{32 \gamma \sqrt{n \sigma_{\max}(X_{\lambda}^{\top}X_{\lambda})}  }, \, \frac{1}{32 \sqrt{p}} \right\}.
        \]
    \end{enumerate}
\end{lemma}

\begin{proof}
    Suppose that $g(\cdot \mid \alpha, \beta, y)$ has the form \eqref{eq:g-probit}.
        In \cite{lee2024fast}, it is shown that, for $\alpha^{(1)}$ and $\alpha^{(2)}$ in $\mathbb{R}^d$ and $\beta^{(1)}$ and $\beta^{(2)}$ in $\mathbb{R}^p$,
        \[
        \begin{aligned}
        &\left\| g(\cdot \mid \alpha^{(1)}, \beta^{(1)}, y) - g(\cdot \mid \alpha^{(2)}, \beta^{(2)}, y) \right\|_{\TV} \\
        \leq & \frac{1}{2} \sqrt{ \left( \begin{array}{c} \alpha^{(1)} - \alpha^{(2)} \\ \beta^{(1)} - \beta^{(2)} \end{array} \right)^{\top} X^{\top} X \left( \begin{array}{c} \alpha^{(1)} - \alpha^{(2)} \\ \beta^{(1)} - \beta^{(2)} \end{array} \right) } \\
        =& \frac{1}{2} \sqrt{ \left( \begin{array}{c} \alpha^{(1)} - \alpha^{(2)} \\ \lambda\beta^{(1)} - \lambda\beta^{(2)} \end{array} \right)^{\top} X_{\lambda}^{\top} X_{\lambda} \left( \begin{array}{c} \alpha^{(1)} - \alpha^{(2)} \\ \lambda \beta^{(1)} - \lambda \beta^{(2)} \end{array} \right) } \\
        \leq& \frac{1}{2} \sqrt{\sigma_{\max}(X_{\lambda}^{\top} X_{\lambda})} \, \sqrt{\|\alpha^{(1)} - \alpha^{(2)}\|_2^2 + \lambda^2 \|\beta^{(1)} - \beta^{(2)}\|_2^2}.
        \end{aligned}
        \]
        This yields \hyperref[A1]{(A1)} with $\epsilon$ and $\delta$ given in 1.

        Suppose instead that $g(\cdot \mid \beta, y)$ has the form \eqref{eq:g-logistic}.
        In \cite{lee2024fast}, it is shown that, for $\alpha^{(1)}$ and $\alpha^{(2)}$ in $\mathbb{R}^d$ and $\beta^{(1)}$ and $\beta^{(2)}$ in $\mathbb{R}^p$,
        \[
        \begin{aligned}
        &\left\| g(\cdot \mid \alpha^{(1)}, \beta^{(1)}, y) - g(\cdot \mid \alpha^{(2)}, \beta^{(2)}, y) \right\|_{\TV} \\
        \leq& \frac{1}{4} \sqrt{\sigma_{\max}(X_{\lambda}^{\top} X_{\lambda})} \, \sqrt{\|\alpha^{(1)} - \alpha^{(2)}\|_2^2 + \lambda^2 \|\beta^{(1)} - \beta^{(2)}\|_2^2}.
        \end{aligned}
        \]
        This yields \hyperref[A1]{(A1)} with $\epsilon$ and $\delta$ given in 2.

        Finally, assume that $g(\cdot \mid \alpha, \beta, y)$ is given in \eqref{eq:g-gaussian}.
        By Lemma \ref{lem:tv-invgauss} in Appendix \ref{app:technical},
        \[
        \begin{aligned}
        &\left\| g(\cdot \mid \alpha^{(1)}, \beta^{(1)}, y) - g(\cdot \mid \alpha^{(2)}, \beta^{(2)}, y) \right\|_{\TV} \\
        \leq & \sqrt{2 \gamma}  \, n^{1/4} \left[ \sum_{i=1}^n (y_i - \alpha^{(1)} - x_i^{\top} \beta^{(1)} - y_i + \alpha^{(2)} + x_i^{\top} \beta^{(2)})^2 \right]^{1/4} \\
        =& \sqrt{2 \gamma}  \, n^{1/4} \left[ \left( \begin{array}{c} \alpha^{(1)} - \alpha^{(2)} \\ \lambda\beta^{(1)} - \lambda\beta^{(2)} \end{array} \right)^{\top} X_{\lambda}^{\top} X_{\lambda} \left( \begin{array}{c} \alpha^{(1)} - \alpha^{(2)} \\ \lambda \beta^{(1)} - \lambda \beta^{(2)} \end{array} \right)  \right]^{1/4} \\
        \leq& \sqrt{2 \gamma}  \, n^{1/4} \left[\sigma_{\max}(X_{\lambda}^{\top}X_{\lambda}) \right]^{1/4} \left(\|\alpha^{(1)} - \alpha^{(2)}\|_2^2 + \lambda^2 \|\beta^{(1)} - \beta^{(2)}\|_2^2\right)^{1/4}.
        \end{aligned}
        \]
        This yields \hyperref[A1]{(A1)} with $\epsilon$ and $\delta$ given in 3.
\end{proof}

The following lemma, which is proved in Appendix \ref{app:proofs}, can be used to verify Condition \hyperref[A2]{(A2)}.(iii).

\begin{lemma} \label{lem:density-ratio}
    Suppose that there exists a number $C \in (0,\infty)$ (which may depend on $y$) such that $\Lik(y \mid \alpha, \beta) \leq C$ for $\alpha \in \mathbb{R}^d$ and $\beta \in \mathbb{R}^p$.
    Assume further that there exist $\ell_0 \in \mathbb{R}$, $L \in [0,\infty)$, and $\eta \in \mathbb{R}^{d+p}$ (all of which may depend on $y$, $\theta$, and $\lambda$) such that, for $\alpha \in \mathbb{R}^d$ and $\beta \in \mathbb{R}^p$, 
        \begin{equation} \label{ine:ell-bound}
        \ell(\alpha, \beta) \leq \ell_0 + \eta^{\top} \left( \begin{array}  {c}
             \alpha \\ \lambda \beta
            \end{array} \right) + \frac{L}{2} (\|\alpha\|_2^2 + \lambda^2 \|\beta\|_2^2).
        \end{equation}
    Then \hyperref[A2]{(A2)}.(iii) holds with
    \[
    D = \log C + \ell_0 + \frac{d}{2} \log \left( \frac{L + \theta^2}{\theta^2} \right) + p \max \left\{ \log \left( \frac{4}{\sqrt{5}-1} \right), \; \log \frac{2 \sqrt{L}}{ \sqrt{5}-1 } \right\}.
    \]
\end{lemma}

\begin{lemma} \label{lem:iso-exmaples}
    Each of the following statements holds.
    \begin{enumerate}
        \item In the data augmentation algorithm for Bayesian probit model, $\Lik(\cdot \mid \alpha, \beta)$ has the form \eqref{eq:lik-probit}.
        Then \hyperref[A2]{(A2)} holds with
        \[
        \begin{aligned}
            D =& n \log 2 + \frac{1}{2} \log \left[ \frac{\sigma_{\max}(X_{\lambda}^{\top}X_{\lambda}) + \theta^2}{\theta^2} \right] + \\
            & \quad \quad p \max \left\{ \log \left( \frac{4}{\sqrt{5}-1} \right),\, \log \frac{2 \sqrt{\sigma_{\max}(X_{\lambda}^{\top}X_{\lambda})}}{\sqrt{5}-1 } \right\}.
        \end{aligned}
        \]

        \item In the data augmentation algorithm for Bayesian logistic model, $\Lik(\cdot \mid \alpha, \beta)$ has the form \eqref{eq:lik-logistic}.
        Then \hyperref[A2]{(A2)} holds with
        \[
        \begin{aligned}
            D =& n \log 2 + \frac{1}{2} \log \left[ \frac{\sigma_{\max}(X_{\lambda}^{\top}X_{\lambda})/4 + \theta^2}{\theta^2} \right] + \\
            & \quad \quad p \max \left\{ \log \left( \frac{4}{\sqrt{5}-1} \right), \, \log \frac{\sqrt{\sigma_{\max}(X_{\lambda}^{\top}X_{\lambda})}}{\sqrt{5}-1} \right\}.
        \end{aligned}
        \]

        \item In the data augmentation algorithm for Gaussian linear model with heteroskedasticity, $\Lik(\cdot \mid \alpha, \beta)$ is given in \eqref{eq:g-gaussian}.
        Then \hyperref[A2]{(A2)} holds with
        \[
        \begin{aligned}
            D =& \gamma \left( \sum_{i=1}^n |y_i| + \frac{n}{2} \right) + \frac{1}{2} \log \left[ \frac{\gamma \, \sigma_{\max}(X_{\lambda}^{\top}X_{\lambda}) + \theta^2 }{\theta^2} \right] + \\
            & \quad \quad p \max \left\{ \log \frac{4}{\sqrt{5}-1}, \, \log \frac{2 \sqrt{\gamma \, \sigma_{\max}(X_{\lambda}^{\top}X_{\lambda})}}{\sqrt{5}-1} \right\}.
        \end{aligned}
        \]
    \end{enumerate}
\end{lemma}

\begin{proof}
    Suppose that $\Lik(\cdot \mid \alpha,  \beta)$ has the form \eqref{eq:lik-probit}.
    It is well-known that \hyperref[A2]{(A2)}.(i) holds \cite{pratt1981concavity}, and \hyperref[A2]{(A2)}.(ii) holds because $\ell(\alpha, \beta)$ is twice differentiable.
    To verify \hyperref[A2]{(A2)}.(iii), note that $\Lik(y \mid \alpha, \beta) \leq 1$.
    Moreover, it holds that $X^{\top} X - \nabla^2 \ell(\alpha, \beta)$ is always positive semi-definite, where $\nabla^2 \ell(\cdot)$ is the Hessian matrix of the function $\ell(\cdot)$; see \cite{lee2024fast}.
    Then, by a Taylor expansion, with $\tilde{\eta} = \nabla \ell(0,0) \in \mathbb{R}^{1+p}$,
    \[
    \begin{aligned}
        \ell(\alpha, \beta) &\leq n \log 2 + \tilde{\eta}^{\top} \left( \begin{array}{c} \alpha \\ \beta \end{array} \right) + \frac{1}{2} \left( \begin{array}{c} \alpha \\ \beta \end{array} \right)^{\top} X^{\top} X \left( \begin{array}{c} \alpha \\ \beta \end{array} \right) \\
        &= n \log 2 + \left( \begin{array}{cccc} \tilde{\eta}_1 & \tilde{\eta}_2/\lambda & \cdots & \tilde{\eta}_{1+d}/\lambda \end{array} \right)^{\top} \left( \begin{array}{c} \alpha \\ \lambda\beta \end{array} \right) + \frac{1}{2} \left( \begin{array}{c} \alpha \\ \lambda\beta \end{array} \right)^{\top} X_{\lambda}^{\top} X_{\lambda} \left( \begin{array}{c} \alpha \\ \lambda\beta \end{array} \right).
    \end{aligned}
    \]
    By Lemma \ref{lem:density-ratio}, \hyperref[A2]{(A2)}.(iii) holds with $D$ given in 1.

    The proof of statement 2. is analogous.

    Suppose that $\Lik(\cdot \mid \alpha, \beta)$ is given in \eqref{eq:g-gaussian}.
    Then 
    \[
    \ell(\alpha, \beta) = - n \log \left( \frac{\gamma}{2} \right) + \gamma \sum_{i=1}^n |y_i - \alpha - x_i^{\top} \beta|.
    \]
    The condition \hyperref[A2]{(A2)}.(i) clearly holds, and \hyperref[A2]{(A2)}.(ii) follows from Lemma \ref{lem:approx}.
    To establish \hyperref[A2]{(A2)}.(iii), note that $\Lik(y \mid \alpha, \beta) \leq (\gamma/2)^n$.
    Moreover,
    \[
    \begin{aligned}
        \ell(\alpha, \beta) &\leq -n \log \left( \frac{\gamma}{2} \right) + \gamma \sum_{i=1}^n \left( |y_i|  + |\alpha + x_i^{\top} \beta| \right) \\
        &\leq -n \log \left( \frac{\gamma}{2} \right) + \gamma \sum_{i=1}^n \left( |y_i|  + \frac{1}{2} |\alpha + x_i^{\top} \beta|^2 + \frac{1}{2} \right) \\
        &= -n \log \left( \frac{\gamma}{2} \right) + \gamma \left( \sum_{i=1}^n |y_i| + \frac{n}{2} \right) + \frac{\gamma}{2} \left( \begin{array}{c} \alpha \\ \lambda\beta \end{array} \right)^{\top} X_{\lambda}^{\top} X_{\lambda} \left( \begin{array}{c} \alpha \\ \lambda\beta \end{array} \right).
    \end{aligned} 
    \]
    By Lemma \ref{lem:density-ratio}, \hyperref[A2]{(A2)}.(iii) holds with $D$ given in 3.
\end{proof}

\section{Proof of Theorem \ref{thm:main}} \label{sec:proof}

\subsection{Close coupling, isoperimetry, and spectral gap}

Similar to \cite{lee2024fast} and multiple recent works, this paper follows the strategy of combining a close coupling condition with an isoperimetric inequality to obtain a convergence bound.

We first introduce the notion of spectral gap.
The spectral gap of the reversible transition density $\Mtd$ is
\[
\begin{aligned}
    \mathrm{Gap}(\Mtd) =& 1 - \sup_\varphi \int_{\mathbb{R}^{d+p}} \langle \Mtd \varphi, \varphi \rangle,
\end{aligned}
\]
where $\langle \Mtd \varphi, \varphi \rangle$ is
\[
\int_{\mathbb{R}^{d+p}} \left[ \int_{\mathbb{R}^{d+p}}  \varphi(\alpha', \beta') \,  \Mtd((\alpha, \beta), (\alpha', \beta')) \, \df (\alpha', \beta') \right] \, \varphi(\alpha, \beta) \, \pi_{A, B \mid Y}(\alpha, \beta \mid y) \,\df (\alpha, \beta),
\]
and the supremum is taken over all functions $\varphi: \mathbb{R}^{d + p}$ such that
\[
\int_{\mathbb{R}^{d+p}} \varphi(\alpha, \beta) \, \pi_{A,B \mid Y}(\alpha, \beta \mid y) \, \df (\alpha, \beta) = 0, \quad \int_{\mathbb{R}^{d+p}} \varphi(\alpha, \beta)^2 \, \pi_{A,B \mid Y}(\alpha, \beta \mid y) \, \df (\alpha, \beta) = 1.
\]

For a generic Markov transition density, its spectral gap is closely related to the rate of convergence of the corresponding Markov chain.
We refer the reader to \cite{roberts1997geometric} and \cite{qin2025convergence} for a detailed examination.
The next lemma is a direct consequence of Theorem 2.1 from \cite{roberts1997geometric}.

\begin{lemma} \label{lem:roberts} \cite{roberts1997geometric}
    Suppose that $\Mtd$ is positive semi-definite in the sense that $\langle \Mtd \varphi, \varphi \rangle$ is always non-negative.
    Then, for $t \geq 1$,
    \[
    \|\omega \, \Mtd^t(\cdot) - \pi_{A,B \mid Y} (\cdot \mid y)\|_{\TV} \leq \tilde{C}(\omega) \, \rho^t,
    \]
    where $\omega(\cdot)$ is any initial density, 
    \[
    \tilde{C}(\omega) = \frac{1}{2} \sqrt{\int_{\mathbb{R}^{d+p}} \left[ \frac{\omega(\alpha, \beta)}{\pi_{A,B \mid Y}(\alpha, \beta \mid y)} - 1 \right]^2 \pi_{A, B \mid Y}(\alpha, \beta \mid y) \, \df (\alpha, \beta)} \, ,
    \]
    and $\rho = 1 - \mathrm{Gap}(\Mtd) \in [0,1]$.
\end{lemma}

To derive Theorem \ref{thm:main}, we bound $\mathrm{Gap}(\Mtd)$ by combining a close coupling condition with an isoperimetric inequality.
In Section \ref{ssec:close}, we establish the following close coupling condition.

\begin{lemma}[Close coupling] \label{lem:close}
    Assume that \hyperref[A1]{(A1)} holds.
    Then, when $\|\alpha^{(1)} - \alpha^{(2)}\|_2^2 + \lambda^2 \|\beta^{(1)} - \beta^{(2)}\|_2^2 < \delta^2$, 
    \[
    \|\Mtd((\alpha^{(1)}, \beta^{(1)}), \cdot) - \Mtd((\alpha^{(2)}, \beta^{(2)}), \cdot)\|_{\TV} \leq 1 - \epsilon.
    \]
\end{lemma}

For a probability measure $\mu$ on a metric space $(\mathsf{X}, \mathrm{dist})$ and a measurable set $\mathsf{A} \subset \mathsf{X}$, let
\[
\mu^+(\mathsf{A}) = \liminf_{r \downarrow 0} \frac{\mu(\mathsf{A}^r) - \mu(\mathsf{A})}{r},
\]
where $\mathsf{A}^r = \{x \in \mathsf{X}: \, \mathrm{dist}(x,y) < r \text{ for some } y \in \mathsf{A}\}$.
Throughout, we set $\mathsf{X} = \mathbb{R}^{d+p}$, and let 
\[
\mathrm{dist}((\alpha^{(1)}, \beta^{(1)}), (\alpha^{(2)}, \beta^{(2)})) = \sqrt{ \|\alpha^{(1)} - \alpha^{(2)}\|_2^2 + \lambda^2 \|\beta^{(1)} - \beta^{(2)}\|_2^2 }
\]
for $\alpha^{(1)}, \alpha^{(2)} \in \mathbb{R}^d$ and $\beta^{(1)}, \beta^{(2)} \in \mathbb{R}^p$.
We say $\mu$ satifies a Cheeger isoperimetric inequality if there exists $C_{\scriptsize\mbox{iso}} > 0$ such that
\[
\mu^+(\mathsf{A}) \geq C_{\scriptsize\mbox{iso}} \min\{ \mu(\mathsf{A}), 1 - \mu(\mathsf{A})\}
\]
whenever $0 < \mu(\mathsf{A}) < 1$.
The greatest $C_{\scriptsize\mbox{iso}}$ for which the isoperimetric inequality holds is called the Cheeger isoperimetric constant of~$\mu$, and will be denoted by $\mbox{Is}(\mu(\cdot))$.
When $\mu$ is a probability density function, $\mbox{Is}(\mu(\cdot))$ denotes the isoperimetric constant of the corresponding distribution.

In Section \ref{ssec:iso}, we establish the following isoperimetric inequality for $\pi_{A, B \mid Y}(\cdot \mid y)$.

\begin{lemma}[Isoperimetric inequality] \label{lem:isoperimetric}
    Assume that \hyperref[A2]{(A2)} holds.
    Then there exists a positive universal constant $C_1$ such that 
    \[
    \mathrm{Is}(\pi_{A, B \mid Y}(\cdot \mid y)) \geq \frac{C_1 \min\{1, \theta\} }{D+1 }.
    \]
\end{lemma}

Lemmas \ref{lem:close} and \ref{lem:isoperimetric} are used to form a bound on the spectral gap of $\Mtd$.
Theorem 18 of \cite{andrieu2024explicit} is a generic result that entails this procedure in a general setting.
From that theorem, we immediately have Lemma \ref{lem:andrieu} below.

\begin{lemma} \label{lem:andrieu} \cite{andrieu2024explicit}
    Suppose that there exist positive numbers $\delta$ and $\epsilon$ such that, when $\|\alpha^{(1)} - \alpha^{(2)}\|_2^2 + \lambda^2 \|\beta^{(1)} - \beta^{(2)}\|_2^2 < \delta^2$, 
    \[
    \|\Mtd((\alpha^{(1)}, \beta^{(1)}), \cdot) - \Mtd((\alpha^{(2)}, \beta^{(2)}), \cdot)\|_{\TV} \leq 1 - \epsilon.
    \]
    Then 
    \[
    \mathrm{Gap}(\Mtd) \geq \frac{1}{32} \epsilon^2 \min \left\{ 1, \frac{1}{4} \delta^2 \, \mathrm{Is}(\pi_{A,B \mid Y}(\cdot \mid y)) ^2 \right\}.
    \]
\end{lemma}

We can now prove Theorem \ref{thm:main}.
\begin{proof} [Proof of Theorem \ref{thm:main}]
    In light of Lemmas \ref{lem:roberts}, \ref{lem:close}, \ref{lem:isoperimetric}, and \ref{lem:andrieu}, it suffices to verify $\Mtd$ is positive semi-definite.
    But it is well-known that the transition densities of data augmentation algorithms are positive semi-definite; see, e.g., \cite{khare2011spectral}.
    Hence, the result follows.
\end{proof}

\subsection{Close coupling} \label{ssec:close}

The goal of this subsection is to establish Lemma \ref{lem:close}.

From Algorithm \ref{alg:da}, we see
\[
\Mtd((\alpha, \beta), (\alpha', \beta')) = \int_{\Z \times (0,\infty)^p} \pi_{A, B \mid Z, T, Y}(\alpha', \beta' \mid z, 1/\xi, y) \, g(z \mid \alpha, \beta, y) \, \tilde{h}(\xi; \beta, \lambda) \, \df(z, \xi),
\]
where $1/\xi = (1/\xi_1, \dots, 1/\xi_p)$,
\[
\tilde{h}(\xi; \beta, \lambda) = \prod_{j=1}^p h(\xi_j ; \beta_j, \lambda), \quad \xi \in (0,\infty)^p,
\]
with
\[
h(u ; \beta_j, \lambda) = \frac{\lambda}{\sqrt{2 \pi}} e^{\lambda |\beta_j|} \, u^{-3/2} \exp \left( - \frac{\beta_j^2 u}{2} - \frac{\lambda^2}{2u} \right) , \quad u \in (0,\infty),
\]
being the density function of the $\mbox{InvGaussian}(\lambda/|\beta_j|, \lambda^2)$ distribution.

The following lemma will be useful for establishing a close coupling condition.

\begin{lemma} \label{lem:close-1}
    Let $\alpha^{(1)}$ and $\alpha^{(2)}$ be elements of $\mathbb{R}^d$, and let $\beta^{(1)}$ and $\beta^{(2)}$ be elements of $\mathbb{R}^p$.
    Then
    \[
    \begin{aligned}
        &\|\Mtd((\alpha^{(1)}, \beta^{(1)}), \cdot) - \Mtd((\alpha^{(2)}, \beta^{(2)}), \cdot)\|_{\TV} \\ \leq & \| \tilde{h}(\cdot; \beta^{(1)}, \lambda) - \tilde{h}(\cdot; \beta^{(2)}, \lambda) \|_{\TV} + \left\| g(\cdot \mid \alpha^{(1)}, \beta^{(1)}, y) - g(\cdot \mid \alpha^{(2)}, \beta^{(2)}, y) \right\|_{\TV} .
    \end{aligned}
    \]
\end{lemma}

\begin{proof}
    It holds that
    \[
    \begin{aligned}
        & 2 \, \|\Mtd((\alpha^{(1)}, \beta^{(1)}), \cdot) - \Mtd((\alpha^{(2)}, \beta^{(2)}), \cdot)\|_{\TV} \\
        =& \int_{\mathbb{R}^{d+p} } \left| \int_{\Z \times (0,\infty)^p} \pi_{A, B \mid Z, T, Y}(\alpha', \beta' \mid z, 1/\xi, y) \left[ \, g(z \mid \alpha^{(1)}, \beta^{(1)}, y) \, \tilde{h}(\xi; \beta^{(1)}, \lambda) - \right. \right. \\
        & \qquad \left. \left. g(z \mid \alpha^{(2)}, \beta^{(2)}, y) \, \tilde{h}(\xi; \beta^{(2)}, \lambda) \right] \, \df (z, \xi) \, \right| \, \df (\alpha', \beta') \\
        \leq& \int_{\mathbb{R}^{d+p} \times \Z \times (0,\infty)^p} \pi_{A, B \mid Z, T, Y}(\alpha', \beta' \mid z, 1/\xi, y) \left| \, g(z \mid \alpha^{(1)}, \beta^{(1)}, y) \, \tilde{h}(\xi; \beta^{(1)}, \lambda) - \right. \\
        & \qquad \left. g(z \mid \alpha^{(2)}, \beta^{(2)}, y) \, \tilde{h}(\xi; \beta^{(2)}, \lambda) \right| \, \df (\alpha', \beta', z, \xi) \\
        =& \int_{\Z \times (0,\infty)^p} \left| \, g(z \mid \alpha^{(1)}, \beta^{(1)}, y) \, \tilde{h}(\xi; \beta^{(1)}, \lambda) -  g(z \mid \alpha^{(2)}, \beta^{(2)}, y) \, \tilde{h}(\xi; \beta^{(2)}, \lambda) \right| \, \df(z, \xi) .
    \end{aligned}
    \]
    The desired result then follows from Lemma \ref{lem:tv-tensor} in Appendix \ref{app:technical}, which is a well-known result on the approximate tensorization of total variation distances.
\end{proof}

We can now prove Lemma \ref{lem:close}, which states that, under \hyperref[A1]{(A1)}, when $\|\alpha^{(1)} - \alpha^{(2)}\|_2^2 + \lambda^2 \|\beta^{(1)} - \beta^{(2)}\|_2^2 < \delta^2$, 
    \[
    \|\Mtd((\alpha^{(1)}, \beta^{(1)}), \cdot) - \Mtd((\alpha^{(2)}, \beta^{(2)}), \cdot)\|_{\TV} \leq 1 - \epsilon.
    \]

\begin{proof}[Proof of Lemma \ref{lem:close}]
    Assume that $\|\alpha^{(1)} - \alpha^{(2)}\|_2^2 + \lambda^2 \|\beta^{(1)} - \beta^{(2)}\|_2^2 < \delta^2$.
    Then $\|\beta^{(1)} - \beta^{(2)}\|_2^{1/2} < \sqrt{\delta/\lambda}$.
    By Lemma \ref{lem:tv-invgauss} in Appendix \ref{app:technical}, which is a technical lemma regarding the total variation distance between products of inverse Gaussian distributions,
    \[
    \| \tilde{h}(\cdot; \beta^{(1)}, \lambda) - \tilde{h}(\cdot; \beta^{(2)}, \lambda) \|_{\TV} \leq \sqrt{2} p^{1/4} \delta^{1/2}.
    \]
    By Lemma \ref{lem:close-1},
    \[
    \begin{aligned}
        &\|\Mtd((\alpha^{(1)}, \beta^{(1)}), \cdot) - \Mtd((\alpha^{(2)}, \beta^{(2)}), \cdot)\|_{\TV} \\
        \leq& \sqrt{2} p^{1/4} \delta^{1/2} + \left\| g(\cdot \mid \alpha^{(1)}, \beta^{(1)}, y) - g(\cdot \mid \alpha^{(2)}, \beta^{(2)}, y) \right\|_{\TV}.
    \end{aligned}
    \]
    The desired result then follows from \hyperref[A1]{(A1)}.
\end{proof}

\subsection{Isoperimetric inequality} \label{ssec:iso}

Taking inspiration from \cite{lee2024fast}, we establish Lemma~\ref{lem:isoperimetric} by comparing to a reference distribution that has good isoperimetric properties.
    
For $\theta > 0$ and $\lambda > 0$, define the probability density function
    \[
    \mu_{\theta, \lambda}(\alpha, \beta) = \frac{\theta^d}{(2\pi)^{d/2}} \exp \left( -\frac{\theta^2 \|\alpha\|_2^2}{2} \right) \, \frac{\lambda^p}{2^p} \, \exp \left( - \lambda \|\beta\|_1 \right), \quad \alpha \in \mathbb{R}^d, \; \beta \in \mathbb{R}^p.
    \]

Because $\mu_{\theta, \lambda}(\alpha, \beta)$ is a product of univariate densities, one can obtain a dimension-independent isoperimetric constant for it.

\begin{lemma} \cite{bobkov1997isoperimetric}
    It holds that
    \[
    \mathrm{Is}(\mu_{\theta, \lambda}(\cdot)) \geq \frac{1}{2 \sqrt{6}} \min\left\{ 1, \sqrt{\frac{2}{\pi}} \theta \right\} .
    \]
\end{lemma}

\begin{proof}
    The Cheeger isoperimetric constants (with respect to the Euclidean distance) of the one-dimensional standard Laplace and normal distributions are, respectively, 1 and $\sqrt{2/\pi}$; see, e.g., Theorem 1.3 of \cite{bobkov1997isoperimetric}.
    Rescaling and applying Remark 5.2 of \cite{bobkov1997isoperimetric} yields the desired result.
\end{proof}

One can then use a comparison result from \cite{milman2012properties} to establish an isoperimetric inequality for $\pi_{A,B \mid Y}(\cdot \mid y)$.

\begin{lemma} \cite{milman2012properties}
    Suppose that \hyperref[A2]{(A2)} holds.
    Then there exists a universal costant $C_1' \in (0,\infty)$ such that
    \[
    \mathrm{Is}(\pi_{A, B \mid Y}(\cdot \mid y)) \geq \frac{C_1'}{D+1}  \mathrm{Is}(\mu_{\theta, \lambda}(\cdot)).
    \]
\end{lemma}

\begin{proof}
    In light of Remark 3.5 of \cite{milman2012properties} (or Lemma 5.4 of \cite{milman2009role}), it suffices to show that $\pi_{A, B \mid Y}(\cdot \mid y)$ can be approximated in total variation by a sequence of probability density functions whose logarithms are concave and twice-differentiable.

    Now,
    \[
    \pi_{A, B \mid Y}(\alpha, \beta \mid y) = \frac{\exp[ - \ell(\alpha, \beta) - \theta^2 \|\alpha\|_2^2/2 -\lambda\|\beta\|_1] }{ \int_{\mathbb{R}^{d+p}} \exp[-\ell(\alpha', \beta') - \theta^2 \|\alpha'\|_2^2/2 - \lambda\|\beta'\|_1] \, \df (\alpha', \beta')  }.
    \]
    Denote the numerator by $g_*(\alpha, \beta)$.
    By Lemma \ref{lem:approx}, the function $\lambda \|\beta\|_1$ can be approximately point-wisely from above by a sequence of convex, twice differentiable functions.
    Then, by assumption, $g_*$ can be approximated point-wisely from below by a sequence of non-negative functions $(g_m)_{m=1}^{\infty}$, where $\log g_m(\cdot)$ is concave and twice differentiable.
    
    For each $m$, let $\hat{\pi}_m(\alpha, \beta) \propto g_m(\alpha, \beta)$ be a probability density function.
    By the dominated convergence theorem,
    \[
    c(g_m) := \int_{\mathbb{R}^{d + p}} g_m(\alpha, \beta) \, \df (\alpha, \beta) \to \int_{\mathbb{R}^{d + p}} g_*(\alpha, \beta) \, \df (\alpha, \beta)
    \]
    as $m \to \infty$.
    Thus, for $m$ large enough, $\hat{\pi}_m(\cdot) = g_m(\cdot)/c(g_m)$ is bounded above by $2 \, \pi_{A,B \mid Y}(\cdot \mid y)$.
    Then, by the dominated convergence theorem, as $m \to \infty$,
    \[
    \| \hat{\pi}_m(\cdot) - \pi_{A, B \mid Y}(\cdot \mid y)\|_{\TV} = \frac{1}{2} \int_{\mathbb{R}^{d + p}} \left| \frac{g_m(\alpha, \beta)}{c(g_m)} - \frac{g_*(\alpha,\beta)}{ \int_{\mathbb{R}^{d+p}} g_*(\alpha', \beta') \, \df (\alpha', \beta') } \right| \, \df (\alpha, \beta) \to 0.
    \]
    That is, $\pi_{A, B \mid Y}(\cdot \mid y)$ can be approximated in total variation by $\hat{\pi}_m(\cdot)$.
\end{proof}

Combining the two lemmas above yields Lemma \ref{lem:isoperimetric}, which states that, under \hyperref[A2]{(A2)}, there is a positive universal constant $C_1$ such that
\[
\mathrm{Is}(\pi_{A, B \mid Y}(\cdot \mid y)) \geq \frac{C_1 \min\{1, \theta\} }{D+1 }.
\]

\appendix

\section{Technical Lemmas} \label{app:technical}

We first present an elementary result involving smooth approximations of the $L^1$ norm.

\begin{lemma} \label{lem:approx}
    Let $\mathrm{L}: \mathbb{R}^q \to \mathbb{R}^{q'}$ be a linear function.
    Then the function $x \mapsto \|\mathrm{L}(x)\|_1$ can be approximated point-wisely from above by a sequence of convex functions that are twice differentiable.
\end{lemma}

\begin{proof}
    For $x \in \mathbb{R}^q$, denote the $j$th component of $\mathrm{L}(x)$ by $\mathrm{L}(x)_j$, $j = 1, \dots, q'$.
    For a positive integer~$m$, let $g_m(x) = \sum_{j=1}^{q'} \psi_m(\mathrm{L}(x)_j)$, where $\psi_m(u) = |u|$ if $|u| > 1/m$, and 
    \[
    \psi_m(u) = - \frac{m^3}{8} u^4 + \frac{3m}{4} u^2 + \frac{3}{8m}
    \]
    if $|u| \leq 1/m$.
    Note that $\psi_m(u)$ is convex and twice differentiable, and $\psi_m(u) \downarrow |u|$ point-wisely as $m \to \infty$.
    Then $g_m(x)$ is convex and twice differentiable, and $g_m(x) \downarrow \|\mathrm{L}(x)\|_1$  point-wisely as $m \to \infty$.
\end{proof}

The following well-known lemma concerns the approximate tensorization of the total variation distance.

\begin{lemma} \label{lem:tv-tensor}
    Let $(\Omega_1, \mathcal{F}_1, \nu_1)$ and $(\Omega_2, \mathcal{F}_2, \nu_2)$ be two $\sigma$-finite measure spaces.
    For $i \in \{1,2\}$, let $g_1^{(i)}: \Omega_1 \to [0,\infty)$ and $g_2^{(i)}: \Omega_2 \to [0,\infty)$ be probability density functions, and let $\tilde{g}^{(i)}(u,v) = g_1^{(i)}(u) \, g_2^{(i)}(v)$ for $(u,v) \in \Omega_1 \times \Omega_2$.
    Then
    \[
    \|\tilde{g}^{(1)}(\cdot) - \tilde{g}^{(2)}(\cdot) \|_{\TV} \leq \|g_1^{(1)}(\cdot) - g_1^{(2)}(\cdot)\|_{\TV} + \|g_2^{(1)}(\cdot) - g_2^{(2)}(\cdot)\|_{\TV}.
    \]
\end{lemma}

\begin{proof}
    The result holds because
    \[
    \begin{aligned}
        2 \, \|\tilde{g}^{(1)}(\cdot) - \tilde{g}^{(2)}(\cdot) \|_{\TV} =& \int_{\Omega_1 \times \Omega_2} |g_1^{(1)}(u) \, g_2^{(1)}(v) - g_1^{(2)}(u) \, g_2^{(2)}(v) | \, \nu_1(\df u) \, \nu_2(\df v) \\
        \leq& \int_{\Omega_1} g_1^{(1)}(u) \int_{\Omega_2} |g_2^{(1)}(v) - g_2^{(2)}(v) | \, \nu_2(\df v) \, \nu_1(\df u) + \\
        &\qquad \int_{\Omega_2} g_2^{(2)}(v) \int_{\Omega_1} |g_1^{(1)}(u) - g_1^{(2)}(u)| \, \nu_1(\df u) \, \nu_2(\df v) \\
        =& 2 \, \|g_2^{(1)}(\cdot) - g_2^{(2)}(\cdot) \|_{\TV} + 2 \, \|g_1^{(1)}(\cdot) - g_1^{(2)}(\cdot)\|_{\TV}.
    \end{aligned}
    \]
\end{proof}

The next lemma concerns the total variation between products of inverse Gaussian distributions.

\begin{lemma} \label{lem:tv-invgauss}
    Let $s^{(1)}$ and $s^{(2)}$ be vectors in $\mathbb{R}^q$.
    Let $c$ be a positive number.
    For $i \in \{1,2\}$ and $x \in \mathbb{R}^q$, let
    \[
    h^{(i)}(x) = \prod_{j=1}^q h(x_j; s^{(i)}_j),
    \]
    where, for $b > 0$ and $u \in \mathbb{R}$,
    \[
    h(u, b) = \frac{c}{\sqrt{2 \pi}} e^{c|b|} \, u^{-3/2} \exp \left( - \frac{ b^2 u}{2} - \frac{c^2}{2 u} \right) \ind(u > 0).
    \]
    That is, $h(\cdot, b)$ is the density function of the $\mathrm{InvGaussian}(c/|b|, c^2)$ distribution.
    Then
    \[
    \|h^{(1)}(\cdot) - h^{(2)}(\cdot) \|_{\TV} \leq \sqrt{2 c} \, q^{1/4} \, \|s^{(1)} - s^{(2)}\|_2^{1/2}. 
    \]
\end{lemma}

\begin{proof}
    For two probability density functions $g_1$ and $g_2$, denote their Kullback-Leibler divergence by $\mathrm{KL}(g_1(\cdot) \, \|\, g_2(\cdot))$, defined in terms of the naturual logarithm.
    We first show the following:
    Let $b$ and $b'$ be real numbers such that $|b| \geq |b'|$.
    Then, 
    \begin{equation} \label{ine:close-kl}
    \mbox{KL}(  h(\cdot , b) \, \| \,  h(\cdot , b'  )) \leq c |b - b'|.
    \end{equation}
    Without loss of generality, assume that $b \neq b'$, which would imply $|b| > 0$.
    Suppose that $U$ is distributed as $h(\cdot, b)$, so that it has expectation $c/|b|$.
    Then, one has
    \[
    \begin{aligned}
        \mbox{KL}( \, h(\cdot , b) \, \| \,  h(\cdot , b') \, ) &= E\left[ c (|b| - |b'|) - \frac{(b^2 - b'^2) U}{2}  \right] \\
        &= c \, (|b| - |b'|) - \frac{c(b^2 - b'^2)}{2|b|} \\
        &\leq c \, (|b| - |b'|) \\
        &\leq c \, |b - b'|,
    \end{aligned}
    \]
    establishing \eqref{ine:close-kl}.

    Let $s^{(1)}$ and $s^{(2)}$ be elements of $\mathbb{R}^q$.
    Let $J_1 = \{j \in \{1,\dots,q\}: |s_j^{(1)}| \geq |s_j^{(2)}|\}$, and let $J_2$ be its complement.
    For $i \in \{1,2\}$, define the probability density functions:
    \[
    h_1^{(i)}((x_j)_{j \in J_1}) = \prod_{j \in J_1} h(x_j, s_j^{(i)}), \quad h_2^{(i)}((x_j)_{j \in J_2}) = \prod_{j \in J_2} h(x_j, s_j^{(i)}).
    \]
    By Lemma \ref{lem:tv-tensor} in Appendix~\ref{app:technical},
    \begin{equation} \label{ine:tv-tensor}
    \left\|h^{(1)}(\cdot) - h^{(2)}(\cdot) \right\|_{\TV} \leq \|h_1^{(1)}(\cdot) - h_1^{(2)}(\cdot) \|_{\TV} + \|h_2^{(1)}(\cdot) - h_2^{(2)}(\cdot) \|_{\TV}.
    \end{equation}
    By Pinsker's ienquality and the tensorization of the KL divergence for product measures,
    \[
    \begin{aligned}
    &\|h_1^{(1)}(\cdot) - h_1^{(2)}(\cdot) \|_{\TV} \leq \sqrt{\frac{1}{2} \sum_{j \in J_1} \mbox{KL} (h(\cdot, s_j^{(1)}) \, \| \, h(\cdot, s_j^{(2)})) }, \\
    &\|h_2^{(1)}(\cdot) - h_2^{(2)}(\cdot) \|_{\TV} \leq \sqrt{\frac{1}{2} \sum_{j \in J_2} \mbox{KL} (h(\cdot, s_j^{(2)}) \, \| \, h(\cdot, s_j^{(1)})) }.
    \end{aligned}
    \]
    Combining these displays with \eqref{ine:close-kl} and \eqref{ine:tv-tensor} yields
    \[
    \begin{aligned}
        \left\|h^{(1)}(\cdot) - h^{(2)}(\cdot) \right\|_{\TV} & \leq \sqrt{\frac{c}{2} \sum_{j \in J_1} |s_j^{(1)} - s_j^{(2)}| } + \sqrt{\frac{c}{2} \sum_{j \in J_2} |s_j^{(1)} - s_j^{(2)}| } \\
        &\leq \sqrt{2 c} \sqrt{\sum_{j=1}^q |s_j^{(1)} - s_j^{(2)}|}.
    \end{aligned}
    \]
    By the Cauchy-Schwarz inequality,
    \[
    \sqrt{\sum_{j=1}^q |s_j^{(1)} - s_j^{(2)}|} \leq \sqrt{\sqrt{q} \sqrt{\sum_{j=1}^q |s_j^{(1)} - s_j^{(2)}|^2 } }.
    \]
    The desired result then follows.
\end{proof}

\section{Technical Proofs} \label{app:proofs}

\subsection{Proof of Proposition \ref{pro:initial}}

We first prove the following lemma.

\begin{lemma} \label{lem:initial}
    Assume that there exists a number $C \in (0,\infty)$ such that $\Lik(y \mid \alpha, \beta) \leq C$ for $\alpha \in \mathbb{R}^d$ and $\beta \in \mathbb{R}^p$.
    Assume further that there exist $\ell_0 \in \mathbb{R}$, $L \in [0,\infty)$, and $\eta \in \mathbb{R}^{d + p}$ such that, for $\alpha \in \mathbb{R}^d$ and $\beta \in \mathbb{R}^p$,
    \begin{equation} \label{ine:l-smooth}
    \ell(\alpha, \beta) \leq \ell_0 + \eta^{\top} \left( \begin{array}{c} \alpha \\ \lambda \beta  \end{array} \right) + \frac{L}{2} (\|\alpha\|_2^2 + \lambda^2 \|\beta\|_2^2).
    \end{equation}
    (The quantities $C$, $\ell_0$, $L$, and $\eta$ may depend on $y$, $\lambda$, and $\theta$.)
    Let
    \[
    \begin{aligned}
        \omega_{\eta, L}(\alpha, \beta) = & \left( \frac{L + \theta^2}{2\pi} \right)^{d/2} \left( \frac{L + 1}{2\pi} \right)^{p/2} \lambda^p \\
        & \quad \quad \exp \left\{ - \frac{1}{2} \left[ \left( \begin{array}{c} \alpha \\ \lambda \beta \end{array} \right) + V_L \eta \right]^{\top}  V_L^{-1}  \left[ \left( \begin{array}{c} \alpha \\ \lambda \beta \end{array} \right) + V_L \eta \right]  \right\},
    \end{aligned}
    \]
    where 
    \[
    V_L = \left( \begin{array}{cc}
        I_d/(L + \theta^2) & 0  \\
        0 & I_p/(L+1)
    \end{array} \right).
    \]
    Then, for $\alpha \in \mathbb{R}^d$ and $\beta \in \mathbb{R}^p$,
    \[
    \begin{aligned}
        &\log \omega_{\eta, L}(\alpha, \beta) - \log \pi_{A,B \mid Y}(\alpha, \beta \mid y) \\
        \leq & \frac{d}{2} \log \left( \frac{L}{\theta^2} + 1 \right) + \frac{p}{2} \log(L+1) - \frac{p}{2} \log (2 \pi) - \frac{1}{2} \eta^{\top} V_L \eta \, + \log C + p \log 2 + \ell_0 + \frac{p}{2}.
    \end{aligned}
    \]
\end{lemma}

\begin{proof}
    Recall that $\ell(\alpha, \beta)$ denotes $- \log f(y \mid \alpha, \beta)$.
    It holds that
    \[
    \begin{aligned}
        - \log \pi_{A, B \mid Y}(\alpha, \beta \mid y) =& \log \int_{\mathbb{R}^{d+p}} f(y \mid \alpha', \beta') \exp \left( - \frac{\theta^2 \|\alpha'\|_2^2}{2} - \lambda \|\beta'\|_1 \right) \df(\alpha', \beta') + \\
        &\quad \quad \ell(\alpha, \beta) + \frac{\theta^2 \|\alpha\|_2^2}{2} + \lambda \|\beta\|_1 \\
        \leq & \log C + \log \int_{\mathbb{R}^{d+p}}  \exp \left( - \frac{\theta^2 \|\alpha'\|_2^2}{2} - \lambda \|\beta'\|_1 \right) \df(\alpha', \beta') + \\
        &\quad \quad \ell_0 + \eta^{\top} \left( \begin{array}{c} \alpha \\ \lambda \beta  \end{array} \right) + \frac{L}{2} (\|\alpha\|_2^2 + \|\lambda \beta\|_2^2) + \frac{\theta^2 \|\alpha\|_2^2}{2} + \frac{\|\lambda \beta\|_2^2}{2} + \frac{p}{2} \\
        =& \log C + \frac{d}{2} \log (2 \pi) - d \log \theta + p \log 2 - p \log \lambda + \\
        &\quad \quad \ell_0 + \eta^{\top} \left( \begin{array}{c} \alpha \\ \lambda \beta  \end{array} \right) + \frac{(L+\theta^2) \|\alpha\|_2^2 + (L+1) \|\lambda\beta\|_2^2}{2} + \frac{p}{2}.
    \end{aligned}
    \]
    Note that we have used the Cauchy-Schwarz inequality.
    On the other hand,
    \[
    \begin{aligned}
        \log \omega_{\eta, L}(\alpha, \beta) =& \frac{d}{2} \log (L + \theta^2) + \frac{p}{2} \log (L+1) - \frac{d+p}{2} \log (2 \pi) + p \log \lambda \\
        & \quad\quad - \frac{(L+\theta^2) \|\alpha\|_2^2}{2} - \frac{(L+1) \|\lambda\beta\|_2^2}{2} - \eta^{\top} \left( \begin{array}{c}
             \alpha  \\
             \lambda \beta
        \end{array} \right) - \frac{1}{2} \eta^{\top} V_L \eta.
    \end{aligned}
    \]
    The desired result then follows.
\end{proof}

We can now prove Proposition \ref{pro:initial}.

\begin{proof} [Proof of Proposition \ref{pro:initial}]
    Consider each of the three models.

    \begin{enumerate}
        \item In the context of Bayesian probit lasso model, $f(y \mid \alpha, \beta ) \leq 1$.
        Moreover, by a Taylor expansion, 
        \[
        \ell(\alpha, \beta) \leq n \log 2 + \tilde{\eta}^{\top} \left( \begin{array}{c}
             \alpha  \\
             \beta
        \end{array} \right) + \frac{1}{2} \left( \begin{array}{c}
             \alpha  \\
             \beta
        \end{array} \right)^{\top} X^{\top}X \left( \begin{array}{c}
             \alpha  \\
             \beta
        \end{array} \right),
        \]
        where $\tilde{\eta} = \nabla \ell(0,0)$.
        See the proof of Lemma \ref{lem:iso-exmaples}.
        Then \eqref{ine:l-smooth} holds with
        \[
        \ell_0 = n \log 2, \quad \eta = \left( \begin{array}{c}
             \tilde{\eta}_1 \\
             \tilde{\eta}_2 / \lambda \\
             \cdots \\
             \tilde{\eta}_{1+p}/\lambda
        \end{array} \right) = - \sqrt{\frac{2}{\pi}} \sum_{i=1}^n (2y_i - 1) \left( \begin{array}{c} 1 \\ x_i/\lambda \end{array} \right), \quad L = \sigma_{\max}(X_{\lambda}^{\top} X_{\lambda}).
        \]
        By Lemma \ref{lem:initial}, 
        \[
        \begin{aligned}
            \frac{\omega_{\eta,L}(\alpha, \beta)}{\pi_{A,B \mid Y}(\alpha, \beta \mid y)} \leq& \exp \left\{ \frac{\log[\sigma_{\max}(X_{\lambda}^{\top}X_{\lambda})/\theta^2 + 1] }{2} + \frac{p \log [\sigma_{\max}(X_{\lambda}^{\top}X_{\lambda}) + 1]}{2} + \right. \\
            & \quad \quad \left. p \log 2 + n \log 2 + \frac{p}{2} \right\}.
        \end{aligned}
        \]
        It follows that
        \[
        \begin{aligned}
            \tilde{C}(\omega_{\eta,L}) & \leq \sup_{\alpha, \beta} \frac{\omega_{\eta,L}(\alpha,\beta)}{\pi_{A,B \mid Y}(\alpha, \beta \mid y)} \\
            &\leq \exp \left\{ \frac{\log[\sigma_{\max}(X_{\lambda}^{\top}X_{\lambda})/\theta^2 + 1] }{2} + \frac{p \log [\sigma_{\max}(X_{\lambda}^{\top}X_{\lambda}) + 1]}{2} + \right. \\
            & \quad \quad \left. p \log 2 + n \log 2 + \frac{p}{2} \right\}.
        \end{aligned}
        \]
        The asymptotic relation \eqref{eq:C-omega-examples} holds.

        \item By an argument analogous to 1., \eqref{eq:C-omega-examples} holds in the context of Bayesian logistic model.

        \item In the context of Bayesian Gaussian model with heteroskedasticity, $f(y \mid \alpha, \beta) \leq (\gamma/2)^n$.
        Moreover, it holds that 
        \[
        \ell(\alpha, \beta) = -n \log \left( \frac{\gamma}{2} \right) + \gamma \left( \sum_{i=1}^n |y_i| + \frac{n}{2} \right) + \frac{\gamma}{2} \left( \begin{array}{c} \alpha \\ \lambda\beta \end{array} \right)^{\top} X_{\lambda}^{\top} X_{\lambda} \left( \begin{array}{c} \alpha \\ \lambda\beta \end{array} \right).
        \]
        See the proof of Lemma \ref{lem:iso-exmaples}.
        Then, by Lemma \ref{lem:initial}, \eqref{eq:C-omega-examples} holds.
    \end{enumerate}
\end{proof}

\subsection{Proof of Lemma \ref{lem:density-ratio}}

We prove the following equivalent result.

\begin{lemma}
    Assume that there exists a number $C \in (0,\infty)$ such that $\Lik(y \mid \alpha, \beta) \leq C$ for $\alpha \in \mathbb{R}^d$ and $\beta \in \mathbb{R}^p$.
    Assume further that there exist $\ell_0 \in \mathbb{R}$, $L \in [0,\infty)$, and $\eta \in \mathbb{R}^{d + p}$ such that, for $\alpha \in \mathbb{R}^d$ and $\beta \in \mathbb{R}^p$,
    \[
    \ell(\alpha, \beta) \leq \ell_0 + \eta^{\top} \left( \begin{array}{c} \alpha \\ \lambda \beta  \end{array} \right) + \frac{L}{2} (\|\alpha\|_2^2 + \lambda^2 \|\beta\|_2^2).
    \]
    (The quantities $C$, $\ell_0$, $L$, and $\eta$ may depend on $y$, $\theta$, and $\lambda$.)
    Let 
    \[
        \mu_{\theta, \lambda}(\alpha, \beta) = \frac{\theta^d}{(2\pi)^{d/2}} \exp \left( -\frac{\theta^2 \|\alpha\|_2^2}{2} \right) \, \frac{\lambda^p}{2^p} \, \exp \left( - \lambda \|\beta\|_1 \right).
    \]
    Then
    \[
    \frac{\pi_{A, B \mid Y}(\alpha, \beta \mid y)}{\mu_{\theta, \lambda}(\alpha, \beta)} \leq  C \,  e^{\ell_0} \left( \frac{L + \theta^2}{\theta^2} \right)^{d/2} \max\left\{ \frac{4}{\sqrt{5} - 1}, \frac{2 \sqrt{L}}{\sqrt{5}-1 } \right\}^p.
    \]
\end{lemma}
\begin{proof}
    For $\alpha \in \mathbb{R}^d$ and $\beta \in \mathbb{R}^p$, it holds that
    \begin{equation} \label{ine:density_ratio}
    \begin{aligned}
        &\frac{\pi_{A, B \mid Y}(\alpha, \beta \mid y)}{\mu_{\theta, \lambda}(\alpha, \beta)} \\
        =& \frac{\Lik(y \mid \alpha, \beta) \, \exp( - \theta^2 \|\alpha\|_2^2/2 -\lambda \|\beta\|_1) \, / \, \int_{\mathbb{R}^{d+p}} \exp[-\ell(\alpha', \beta') - \theta^2 \|\alpha'\|_2^2/2 - \lambda\|\beta'\|_1] \, \df (\alpha', \beta') }{ \theta^d \, (2 \pi)^{-d/2} \,  (\lambda/2)^p \exp( - \theta^2 \|\alpha\|_2^2/2 -  \lambda\|\beta\|_1 )} \\
        \leq& \left( \frac{2}{\lambda} \right)^p \left( \frac{2 \pi}{\theta^2} \right)^{d/2} C \,  e^{\ell_0} \times \\
        & \quad \quad \left\{ \int_{\mathbb{R}^{d+p}} \exp \left[ - \eta^{\top} \left( \begin{array}{c} \alpha' \\ \lambda \beta' \end{array} \right) - \frac{L}{2} (\|\alpha'\|_2^2 + \lambda^2 \|\beta'\|_2^2) - \frac{\theta^2 \|\alpha'\|_2^2}{2} - \lambda \|\beta'\|_1 \right] \, \df (\alpha', \beta') \right\}^{-1} \\
        =& \left( \frac{2}{\lambda} \right)^p \left( \frac{2 \pi}{\theta^2} \right)^{d/2} C e^{\ell_0} \left( \frac{L + \theta^2}{ 2 \pi } \right)^{d/2} \exp \left[ - \sum_{j=1}^d \frac{\eta_j^2}{2(L + \theta^2)} \right]  \times \\
        & \quad \quad \left[  \prod_{j=d+1}^{d+p} \lambda^{-1} \int_{-\infty}^{\infty} \exp \left( -\eta_j u - \frac{L}{2} u^2 - |u|  \right) \, \df u \right]^{-1} \\
        \leq & 2^p \left( \frac{L + \theta^2}{\theta^2} \right)^{d/2}  C e^{\ell_0} \left[ \prod_{j=d+1}^{d+p} \int_{-\infty}^{\infty} \exp \left( -\eta_j u - \frac{L}{2} u^2 - |u|  \right) \, \df u \right]^{-1},
    \end{aligned}
    \end{equation}
    where $u = \lambda \beta_j$ in the univariate integrals.
    For $j \in \{d+1,\dots,d+p\}$, 
    \begin{equation} \label{ine:int_pi}
    \begin{aligned}
        & \int_{-\infty}^{\infty} \exp \left( -\eta_j u - \frac{L}{2} u^2 - |u|  \right) \, \df u \\
        =& \int_0^{\infty} \exp \left( -\eta_j u - \frac{L}{2} u^2 - u  \right) \, \df u + \int_{-\infty}^0  \exp \left( -\eta_j u - \frac{L}{2} u^2 + u  \right) \, \df u \\
        =& \sqrt{\frac{2 \pi}{L}} \left[ e^{(1 + \eta_j)^2/(2L)}\int_0^{\infty} \phi \left(v + \frac{1 + \eta_j}{\sqrt{L}} \right) \, \df v + e^{(1-\eta_j)^2/(2L)} \int_{-\infty}^0 \phi\left( v - \frac{1 - \eta_j}{\sqrt{L}} \right) \, \df v \right] \\
        =& \frac{1}{\sqrt{L}} \left[ \frac{\Phi(-(1 + \eta_j)/\sqrt{L})}{\phi((1+ \eta_j)/\sqrt{L})} + \frac{\Phi(-(1-\eta_j)/\sqrt{L})}{\phi((1 - \eta_j)/\sqrt{L})}\right],
    \end{aligned}
    \end{equation}
    where the third line is obtained by letting $v = \sqrt{L} u$ in the integrations.
    
    The Mill's ratio $\Phi(-u) / \phi(u)$ is lower bounded by $\sqrt{\pi/2}$ whenever $u \leq 0$.
    Thus, if $\eta_j \leq - 1$ or $\eta_j \geq 1$, then at least one of the two Mill's ratios in the last line of \eqref{ine:int_pi} is lower bounded by $\sqrt{\pi / 2}$, and
    \begin{equation} \nonumber
    \int_{-\infty}^{\infty} \exp \left( -\eta_j u - \frac{L}{2} u^2 -  |u|  \right) \, \df u  \geq \sqrt{\frac{\pi}{2 L}}.
    \end{equation}
    Assume now that $-1 < \eta_j < 1$, so
    \begin{equation} \label{ine:lambda_gj}
    0 < 1 + \eta_j < 2, \quad 0 < 1 - \eta_j < 2.
    \end{equation}
    By \cite{birnbaum1942}, for $u > 0$, 
    \[
    \frac{\Phi(-u)}{\phi(u)} \geq \frac{\sqrt{4 + u^2} - u}{2} \geq \min\left\{ \frac{\sqrt{5}-1}{2u}, \frac{\sqrt{5}-1}{2} \right\}.
    \]
    Hence, by \eqref{ine:int_pi}, when $-1 < \eta_j < 1$,
    \[
    \begin{aligned}
        &\int_{-\infty}^{\infty} \exp \left( -\eta_j u - \frac{L}{2} u^2 -  |u|  \right) \, \df u \\
        \geq& \frac{\sqrt{5}-1}{2} \frac{1}{\sqrt{L}} \left( \min \left\{ \frac{\sqrt{L}}{1 + \eta_j}, 1 \right\} + \min \left\{ \frac{\sqrt{L}}{1 - \eta_j}, 1 \right\} \right) \\
         \stackrel{\eqref{ine:lambda_gj}}{\geq} &  \frac{\sqrt{5}-1}{2}  \min \left\{ 1, \frac{2}{\sqrt{L}} \right\}  .
    \end{aligned}
    \]
    In summary, for a general $\eta_j$,
    \begin{equation} \label{ine:int_pi_1}
        \begin{aligned}
            \int_{-\infty}^{\infty} \exp \left( -\eta_j u - \frac{L}{2} u^2 - |u|  \right) \, \df u &\geq \min\left\{ \frac{\sqrt{5}-1}{2}, \frac{\sqrt{5}-1}{\sqrt{L}}, \sqrt{\frac{\pi}{2L}} \right\} \\
            &= \min \left\{ \frac{\sqrt{5}-1}{2}, \frac{\sqrt{5}-1}{\sqrt{L}} \right\}.
        \end{aligned}
    \end{equation} 

    Combining \eqref{ine:density_ratio} and \eqref{ine:int_pi_1} yields
    \[
    \frac{\pi_{A, B \mid Y}(\alpha, \beta \mid y)}{\mu_{\theta, \lambda}(\alpha, \beta)} \leq 2^p \left( \frac{L + \theta^2}{\theta^2} \right)^{d/2} C \,  e^{\ell_0} \max\left\{ \frac{2}{\sqrt{5} - 1}, \frac{\sqrt{L}}{\sqrt{5}-1} \right\}^p.
    \]
    The desired result follows.
\end{proof}

\bibliographystyle{plain}
\bibliography{qinbib_lasso}

\begin{thebibliography}{10}

\bibitem{albert1993bayesian}
James~H Albert and Siddhartha Chib.
\newblock Bayesian analysis of binary and polychotomous response data.
\newblock {\em Journal of the American Statistical Association}, 88(422):669--679, 1993.

\bibitem{andrews1974scale}
David~F Andrews and Colin~L Mallows.
\newblock Scale mixtures of normal distributions.
\newblock {\em Journal of the Royal Statistical Society: Series B}, 36(1):99--102, 1974.

\bibitem{andrieu2024explicit}
Christophe Andrieu, Anthony Lee, Sam Power, and Andi~Q Wang.
\newblock {Explicit convergence bounds for Metropolis Markov chains: Isoperimetry, spectral gaps and profiles}.
\newblock {\em Annals of Applied Probability}, 34(4):4022--4071, 2024.

\bibitem{ascolani2025mixing}
Filippo Ascolani and Giacomo Zanella.
\newblock {Mixing times of data-augmentation Gibbs samplers for high-dimensional probit regression}.
\newblock arXiv preprint, 2025.

\bibitem{birnbaum1942}
Z.~W. Birnbaum.
\newblock {An inequality for Mill's ratio}.
\newblock {\em Annals of Mathematical Statistics}, 13:245--246, 1942.

\bibitem{bobkov1997isoperimetric}
Sergey~G Bobkov and Christian Houdr{\'e}.
\newblock Isoperimetric constants for product probability measures.
\newblock {\em Annals of Probability}, pages 184--205, 1997.

\bibitem{chakraborty2016convergence}
Saptarshi Chakraborty and Kshitij Khare.
\newblock Convergence properties of {G}ibbs samplers for {B}ayesian probit regression with proper priors.
\newblock {\em Electronic Journal of Statistics}, 11(1):177--210, 2017.

\bibitem{chewi2025log}
Sinho Chewi.
\newblock {Log-concave Sampling}.
\newblock Book draft, 2025.

\bibitem{choi2013analysis}
Hee~Min Choi and James~P Hobert.
\newblock Analysis of {MCMC} algorithms for {B}ayesian linear regression with {L}aplace errors.
\newblock {\em Journal of Multivariate Analysis}, 117:32--40, 2013.

\bibitem{choi2013polya}
Hee~Min Choi and James~P Hobert.
\newblock {The Polya-Gamma Gibbs sampler for Bayesian logistic regression is uniformly ergodic}.
\newblock {\em Electronic Journal of Statistics}, 7:2054--2064, 2013.

\bibitem{dwivedi2019log}
Raaz Dwivedi, Yuansi Chen, Martin~J Wainwright, and Bin Yu.
\newblock {Log-concave sampling: Metropolis-Hastings algorithms are fast}.
\newblock {\em Journal of Machine Learning Research}, 20(183):1--42, 2019.

\bibitem{dyer1991random}
Martin Dyer, Alan Frieze, and Ravi Kannan.
\newblock A random polynomial-time algorithm for approximating the volume of convex bodies.
\newblock {\em Journal of the Association for Computing Machinery}, 38(1):1--17, 1991.

\bibitem{figueiredo2003adaptive}
M{\'a}rio~AT Figueiredo.
\newblock Adaptive sparseness for supervised learning.
\newblock {\em IEEE Transactions on Pattern Analysis and Machine Intelligence}, 25(9):1150--1159, 2003.

\bibitem{goyal2025mixing}
Alexander Goyal, George Deligiannidis, and Nikolas Kantas.
\newblock {Mixing time bounds for the Gibbs sampler under isoperimetry}.
\newblock arXiv preprint, 2025.

\bibitem{horn2012matrix}
Roger~A Horn and Charles~R Johnson.
\newblock {\em Matrix Analysis}.
\newblock Cambridge University Press, 2nd edition, 2012.

\bibitem{jerrum1988conductance}
Mark Jerrum and Alistair Sinclair.
\newblock {Conductance and the rapid mixing property for Markov chains: the approximation of permanent resolved}.
\newblock In {\em Proceedings of the Twentieth Annual ACM Symposium on Theory of Computing}, pages 235--244, 1988.

\bibitem{johndrow2016mcmc}
James~E. Johndrow, Aaron Smith, Natesh Pillai, and David~B. Dunson.
\newblock {MCMC} for imbalanced categorical data.
\newblock {\em Journal of the American Statistical Association}, 114:1394--1403, 2019.

\bibitem{johnstone2001distribution}
Iain~M Johnstone.
\newblock On the distribution of the largest eigenvalue in principal components analysis.
\newblock {\em Annals of Statistics}, 29(2):295--327, 2001.

\bibitem{khare2011spectral}
Kshitij Khare and James~P Hobert.
\newblock A spectral analytic comparison of trace-class data augmentation algorithms and their sandwich variants.
\newblock {\em Annals of Statistics}, 39(5):2585--2606, 2011.

\bibitem{khare2013geometric}
Kshitij Khare and James~P Hobert.
\newblock {Geometric ergodicity of the Bayesian lasso}.
\newblock {\em Electronic Journal of Statistics}, 7:2150--2163, 2013.

\bibitem{kwon2024phase}
Youngwoo Kwon, Qian Qin, Guanyang Wang, and Yuchen Wei.
\newblock {A phase transition in sampling from Restricted Boltzmann Machines}.
\newblock {\em Annals of Applied Probability, {\rm to appear}}, 2025+.

\bibitem{lawler1988bounds}
Gregory~F Lawler and Alan~D Sokal.
\newblock Bounds on the $l^2$ spectrum for {M}arkov chains and {M}arkov processes: {A} generalization of {C}heeger’s inequality.
\newblock {\em Transactions of the American Mathematical Society}, 309(2):557--580, 1988.

\bibitem{lee2024fast}
Holden Lee and Kexin Zhang.
\newblock Fast mixing of data augmentation algorithms: {B}ayesian probit, logit, and lasso regression.
\newblock arXiv preprint, 2024.

\bibitem{lovasz1993random}
L{\'a}szl{\'o} Lov{\'a}sz and Mikl{\'o}s Simonovits.
\newblock Random walks in a convex body and an improved volume algorithm.
\newblock {\em Random structures \& algorithms}, 4(4):359--412, 1993.

\bibitem{milman2009role}
Emanuel Milman.
\newblock On the role of convexity in isoperimetry, spectral gap and concentration.
\newblock {\em Inventiones Mathematicae}, 177(1):1--43, 2009.

\bibitem{milman2012properties}
Emanuel Milman.
\newblock Properties of isoperimetric, functional and transport-entropy inequalities via concentration.
\newblock {\em Probability Theory and Related Fields}, 152(3):475--507, 2012.

\bibitem{park2008bayesian}
Trevor Park and George Casella.
\newblock The {B}ayesian lasso.
\newblock {\em Journal of the American Statistical Association}, 103(482):681--686, 2008.

\bibitem{polson2013bayesian}
Nicholas~G Polson, James~G Scott, and Jesse Windle.
\newblock {Bayesian inference for logistic models using P{\'o}lya--Gamma latent variables}.
\newblock {\em Journal of the American Statistical Association}, 108(504):1339--1349, 2013.

\bibitem{pratt1981concavity}
John~W Pratt.
\newblock Concavity of the log likelihood.
\newblock {\em Journal of the American Statistical Association}, 76(373):103--106, 1981.

\bibitem{qin2025convergence}
Qian Qin.
\newblock {Convergence bounds for Monte Carlo Markov chains}.
\newblock In Radu~V Craiu, Dootika Vats, Galin~L Jones, Steve Brooks, Andrew Gelman, and Xiao-Li Meng, editors, {\em Handbook of Markov chain Monte Carlo}. to appear, 2025+.

\bibitem{qin2019convergence}
Qian Qin and James~P Hobert.
\newblock Convergence complexity analysis of {A}lbert and {C}hib's algorithm for {B}ayesian probit regression.
\newblock {\em Annals of Statistics}, 47:2320--2347, 2019.

\bibitem{qin2020wasserstein}
Qian Qin and James~P Hobert.
\newblock {Wasserstein-based methods for convergence complexity analysis of MCMC with applications}.
\newblock {\em Annals of Applied Probability}, 32:124--166, 2022.

\bibitem{rajaratnam2015mcmc}
Bala Rajaratnam and Doug Sparks.
\newblock {MCMC-based inference in the era of big data: A fundamental analysis of the convergence complexity of high-dimensional chains}.
\newblock arXiv preprint, 2015.

\bibitem{rajaratnam2015scalable}
Bala Rajaratnam, Doug Sparks, Kshitij Khare, and Liyuan Zhang.
\newblock {Scalable Bayesian shrinkage and uncertainty quantification for high-dimensional regression}.
\newblock arXiv preprint, 2015.

\bibitem{roberts1997geometric}
Gareth~O Roberts and Jeffrey~S Rosenthal.
\newblock Geometric ergodicity and hybrid {M}arkov chains.
\newblock {\em Electronic Communications in Probability}, 2(2):13--25, 1997.

\bibitem{roy:hobe:2007}
Vivekananda Roy and James~P. Hobert.
\newblock {C}onvergence rates and asymptotic standard errors for {M}arkov chain {M}onte {C}arlo algorithms for {B}ayesian probit regression.
\newblock {\em Journal of the Royal Statistical Society, Series B}, 69:607--623, 2007.

\bibitem{roy2010monte}
Vivekananda Roy and James~P Hobert.
\newblock On {M}onte {C}arlo methods for {B}ayesian multivariate regression models with heavy-tailed errors.
\newblock {\em Journal of Multivariate Analysis}, 101(5):1190--1202, 2010.

\bibitem{tanner1987calculation}
Martin~A Tanner and Wing~Hung Wong.
\newblock The calculation of posterior distributions by data augmentation (with discussion).
\newblock {\em Journal of the American Statistical Association}, 82(398):528--540, 1987.

\bibitem{tibshirani1996regression}
Robert Tibshirani.
\newblock Regression shrinkage and selection via the lasso.
\newblock {\em Journal of the Royal Statistical Society Series B}, 58(1):267--288, 1996.

\bibitem{tierney1994markov}
Luke Tierney.
\newblock Markov chains for exploring posterior distributions.
\newblock {\em Annals of Statistics}, 22(4):1701--1728, 1994.

\bibitem{van2008high}
Sara~A van~de Geer.
\newblock {High-dimensional generalized linear models and the lasso}.
\newblock {\em Annals of Statistics}, 36(2):614 -- 645, 2008.

\bibitem{van2001art}
David~A van Dyk and Xiao-Li Meng.
\newblock The art of data augmentation (with discussion).
\newblock {\em Journal of Computational and Graphical Statistics}, 10(1):1--50, 2001.

\bibitem{vempala2019rapid}
Santosh Vempala and Andre Wibisono.
\newblock {Rapid convergence of the unadjusted Langevin algorithm: Isoperimetry suffices}.
\newblock {\em Advances in Neural Information Processing Systems}, 32, 2019.

\bibitem{wadia2024mixing}
Neha~S Wadia.
\newblock A mixing time bound for gibbs sampling from log-smooth log-concave distributions.
\newblock arXiv preprint, 2024.

\bibitem{wu2022minimax}
Keru Wu, Scott Schmidler, and Yuansi Chen.
\newblock {Minimax mixing time of the Metropolis-adjusted Langevin algorithm for log-concave sampling}.
\newblock {\em Journal of Machine Learning Research}, 23(270):1--63, 2022.

\bibitem{zhao2006model}
Peng Zhao and Bin Yu.
\newblock On model selection consistency of lasso.
\newblock {\em The Journal of Machine Learning Research}, 7:2541--2563, 2006.

\bibitem{zhou2025polynomial}
Quan Zhou.
\newblock Polynomial mixing times of simulated tempering for mixture targets by conductance decomposition.
\newblock arXiv preprint, 2025.

\end{thebibliography}
\end{document}